\documentclass{article}
\usepackage[toc,page]{appendix}
\usepackage[utf8]{inputenc}
\usepackage[english]{babel}
\DeclareMathAlphabet{\mathscr}{OT1}{pzc}{m}{it} 
\usepackage{amsmath}
\usepackage{bbold}
\usepackage{dsfont} 
\usepackage{amssymb}
\usepackage[T1]{fontenc}
\usepackage{mathtools}   
\usepackage{amsfonts}
\usepackage{bbm}
\usepackage{mathrsfs}  
\usepackage{graphicx}
\usepackage{bigints}
\usepackage{relsize}
\usepackage{enumitem}
\usepackage{relsize}
\usepackage{extpfeil}
\usepackage[colorinlistoftodos]{todonotes}
\usepackage{amssymb,amsmath,amsthm}
\numberwithin{equation}{section}
\usepackage{authblk}

\newtheorem{theorem}{Theorem}[section]
\newtheorem{notation}[theorem]{Notation}
\newtheorem{lemma}[theorem]{Lemma}

\newtheorem{proposition}[theorem]{Proposition}
\newtheorem{corollary}[theorem]{Corollary}
\newtheorem{definition}[theorem]{Definition}

\newtheorem{assumption}[theorem]{Assumption}
\newtheorem{remark}[theorem]{Remark}
\newenvironment{prooff}[1]{\begin{trivlist}
\item {\it \bf Proof}\quad} {\qed\end{trivlist}}
\usepackage{amsmath,amssymb}
\makeatletter
\newsavebox\myboxA
\newsavebox\myboxB
\newlength\mylenA
\newcommand*\xoverline[2][0.75]{%
    \sbox{\myboxA}{$\m@th#2$}%
    \setbox\myboxB\null
    \ht\myboxB=\ht\myboxA%
    \dp\myboxB=\dp\myboxA%
    \wd\myboxB=#1\wd\myboxA
    \sbox\myboxB{$\m@th\overline{\copy\myboxB}$}
    \setlength\mylenA{\the\wd\myboxA}
    \addtolength\mylenA{-\the\wd\myboxB}%
    \ifdim\wd\myboxB<\wd\myboxA%
       \rlap{\hskip 0.5\mylenA\usebox\myboxB}{\usebox\myboxA}%
    \else
        \hskip -0.5\mylenA\rlap{\usebox\myboxA}{\hskip 0.5\mylenA\usebox\myboxB}%
    \fi}
\makeatother

\def \ep{\hbox{ }\hfill{ $\square$   } }

\title{Controlled diffusion Mean Field Games with common noise, and McKean-Vlasov second order backward SDEs}

\author{
Adrien BARRASSO \thanks{Ecole Polytechnique, CMAP boulevard des Mar\'echaux, F-91120 Palaiseau, France.
E-mail: \sf adrien.barrasso@polytechnique.edu. Research supported by ANR project PACMAN, and the joint lab FiME.}
\qquad\quad
Nizar TOUZI \thanks{Ecole Polytechnique, CMAP boulevard des Mar\'echaux, F-91120 Palaiseau, France.
	E-mail: \sf nizar.touzi@polytechnique.edu. Research supported by ANR project PACMAN, and the Chaires FiME-FDD and Financial Risks of the Louis Bachelier Institute. \\}}

\date{\today}

\begin{document}
\maketitle

{\bf Abstract.}
We consider a mean field game with common noise  in which the diffusion coefficients may be controlled. We prove existence of a weak relaxed solution under some continuity conditions on the coefficients. We then show that, when there is no common noise, the solution of this mean field game is characterized  by a McKean-Vlasov type second order backward SDE.

\section{Introduction}

In this paper we consider a mean field game with common noise  in which the diffusion coefficients may be controlled. Mean field games have been introduced by Lasry \& Lions \cite{lasry_lions_2007}, and Huang, Malhamé \& Caines \cite{huang_malhame_caines}, and generated a very extended literature. In the present paper, we address an extension which allows for diffusion control and the presence of common noise.

The problem is defined as a Nash equilibrium within a crowd of players who solve, given a fixed random measure $M$, the individual maximization problem
 \begin{equation}\label{IntroCost}
 \sup_{\alpha}\mathbbm{E}\left[\xi(X^{\alpha,M}) + \int_0^Tf_r(X^{\alpha,M},\alpha_r,M)dr\right],
 \end{equation}
where $X^{\alpha,M}$ is the solution of the controlled non-Markovian SDE 
\begin{equation}\label{IntroSDE}
dX^{\alpha,M}_t= b_t(X^{\alpha,M},\alpha_t,M)dt + \sigma^1_t(X^{\alpha,M},\alpha_t,M)dW^1_t+\sigma^0_t(X^{\alpha,M},\alpha_t,M)dW^0_,
\end{equation}
and $\alpha$ is the control process of a typical player. Here, $X^{\alpha,M}$ is the state process of a typical player, with dynamics controlled by $\alpha$, and governed by the \textit{individual noise} $W^1$ and the \textit{common noise} $W^0$. The individual noise $W^1$ only impacts the dynamics of one specific player, while the common noise $W^0$ impacts the dynamics of all players. 

The coefficients of the state equation depend on the random distribution $M$, which represents a distribution on the canonical space of the state process conditional on the common noise $W^0$, and is intended to model the empirical distribution of the states of the interacting crowd of players.

A solution of the mean field game is then a random measure $M$ such that the corresponding optimal diffusion $X^{*,M}$ induced by the problem \eqref{IntroCost} satisfies:
\begin{equation}\label{IntroEquilib}
	M = \mathbbm{P}\circ (X^{*,M}|W^0)^{-1}\quad \text{a.s,}
\end{equation}
where $\mathbbm{P}\circ (X^{*,M}|W^0)^{-1}$ denotes the conditional law of $X^{*,M}$ given $W^0$. 

We prove existence of a weak relaxed solution of this problem under some continuity conditions on the coefficients. By \textit{weak solution} we mean that we work with a controlled martingale problem instead of a controlled SDE intended in the strong sense, and that we find a weaker fixed point of type 	$M = \mathbbm{P}\circ (X^{*,M}|W^0,M)^{-1}$ a.s. instead of \eqref{IntroEquilib}, a notion introduced by Carmona, Delarue \& Lacker \cite{LackerCommonNoise}. By \textit{relaxed solution} we mean that we allow relaxed controls, also called mixed strategies, which is the standard framework in stochastic control theory in order to guarantee existence of optimal controls, see Hausmann \cite{haussmann1976general} and El Karoui, Jeanblanc \& N'Guyen  \cite{nicole1987compactification}. If the control process $\alpha$ takes values in a subset $A$ of a finite dimensional space, then relaxed controls $q$ take values $q_t$ in the space $\mathfrak{M}_+^1(A)$ of probability measures on $A$. 

In the relaxed formulation, the state process $X^{q,M}$ is controlled by the relaxed control $q$, and the cost functional takes the relaxed form 
$$
\mathbbm{E}\left[\xi(X^{q,M}) + \int_0^T\int_Af_r(X^{q,M},a,M)q_r(da)dr\right].
$$
The first main result of this paper is the existence of a weak relaxed solution of the mean field game in the context where the state dynamics exhibit both common noise and controlled diffusion coefficients.

The second part of the paper specializes to the no common noise setting. In this context, our second main result is a characterization of the solution of this mean field game by means of a McKean-Vlasov second order backward SDE of the form
\begin{equation}\label{IntroMkV2BSDE}
	Y_t = \xi +\int_t^T F_r(X,Z_r,\hat{\sigma}_r^2,m)dr - \int_t^TZ_rdX_r + \, U_T-U_t,\quad t\in[0,T],\quad \mathcal{P}^m-\text{q.s.}
\end{equation}
whose precise meaning will be made explicit in Section \ref{S2BSDE}. This extends the previous results by Carmona \& Delarue \cite{CarmonaDelarueI,CarmonaDelarueII} characterizing the solution of a mean field game by McKean-Vlasov backward SDEs in the  uncontrolled diffusion setting. We believe that the present paper is the first instance of interest in such McKean-Vlasov second order backward SDEs.

\vspace{5mm}

\noindent {\bf Literature review.} Mean field games have been introduced by the pioneering
works of Lasry \& Lions \cite{lasry_lions_2007}, and Huang, Malhamé \& Caines \cite{huang_malhame_caines}. Their works were the first ones to consider the limit of a symmetric game of $N$ players when $N$ tends to infinity, and to link it to a fixed point problem of Mc-Kean Vlasov type, which in its most simple form may be described as follows.
\begin{enumerate}
	\item For any probability measure $m$ on the space of continuous paths, find the optimal control $\alpha^{m}$ which minimizes 
	the cost functional 
	\begin{equation}\label{IntroCost2}
	\mathbbm{E}\left[g(X^{\alpha}_T) + \int_0^Tf_r(X^{\alpha}_r,\alpha_r,m)dr\right],
	\end{equation}
	where $X^{\alpha}$ is the controlled diffusion  of dynamics 
	\begin{equation}\label{IntroSDE2}
	dX^{\alpha}_t= \alpha_tdt + dW_t.
	\end{equation}
	\item Find a equilibrium measure verifying $m^*=\mathcal{L}\left(X^{\alpha^{m^*}}\right)$.
\end{enumerate}
The idea being that $m^*$ models the behavior of a population of individuals. Each one of these individuals  controls a diffusion of type \eqref{IntroSDE2}, where $W$ is a Brownian motion "observed" only by this specific individual  and optimizes  the cost \eqref{IntroCost2}.

During the following decade, this topic generated a huge literature with results based  on PDE methods on one hand (see for instance Lasry \& Lions \cite{lasry_lions_2007}), and on probabilistic methods on the other hand, namely through McKean-Vlasov forward-backward SDEs, see Carmona \& Delarue \cite{CarmonaDelarueI} for an overview.

The extension of mean field games to the common noise situation (i.e with an additional noise $W^0$ in \eqref{IntroSDE2}) was addressed recently, motivated by a strong need from applications so as to introduce a source of randomness observed by all players. One may for example refer to \cite{CarmonaDelarueII}.

The first part of the present paper is in the continuity of 
a recent sequence of papers due to R. Carmona, F. Delarue and D. Lacker. In particular, \cite{LackerVolatility} proves existence of a weak relaxed solution for a MFG with controlled diffusion coefficient but without common noise under merely continuity assumptions on the coefficients, and \cite{LackerCommonNoise} shows existence of a weak solution of an MFG with common noise but without control in the diffusion coefficient, under similar continuity assumptions on the coefficients. The present paper fills the gap between these two works, by extending this existence result in the situation with common noise, and allowing for diffusion control.

While MFGs with a control in the drift are connected to McKean-Vlasov backward SDEs, one naturally expect that the control in the diffusion coefficient will in some way link the MFG to the second order extension of backward SDEs. The latter is a notion of Sobolev type solution for path-dependent PDEs, introduced by Soner, Touzi \& Zhang \cite{soner_touzi_zhang1} as a representation of diffusion control problems (in contrast with backward SDEs which are related to drift control). 
A first existence result was obtained in \cite{soner_touzi_zhang2}, and such second order backward SDEs proved very useful to study fully non linear second order PDEs, as an extension of the links between backward SDEs and semi-linear PDEs, see \cite{ektz,ekren_touzi_zhang}. We also refer to Possama\"{\i}, Tan \& Zhou \cite{ptz} for a more general existence result, and to Lin, Ren, Touzi \& Yang \cite{lrty} for the extension to a random terminal time. 
\\


The paper is organized in two parts. Sections \ref{SMFG} and \ref{MFGSol} concern Mean Field Games with common noise and controlled diffusion coefficient; Sections \ref{SMkV2BSDE}, \ref{S2BSDE} and \ref{S6} develop the links between MFGs and McKean-Vlasov second order backward SDEs.

Section \ref{SMFG} provides the precise formulation of our mean field game, see in particular Definition \ref{weaksol}.  Section \ref{MFGSol} is devoted to the proof of existence of a weak relaxed solution (see Theorem \ref{T_existence}) under Assumption \ref{H_existence}. The proof is divided in three parts. We start by showing some preliminary topological results in Subsection \ref{S3.2}. Then, in Subsection \ref{sec:discretenoise}, we introduce as in Carmona Delarue \& Lacker \cite{LackerCommonNoise} the notion of discretized strong equilibiria (see Definition \ref{D_discret}) and prove existence of such equilibria, see Proposition \ref{P_discret}. Finally, in Subsection \ref{S3.4}, we conclude the proof of existence of a weak relaxed solution of the MFG by considering the limit of discretized strong equilibria.
 \\
 In Section \ref{SMkV2BSDE}, we introduce the notion of McKean-Vlasov 2BSDE (see Definition \ref{MkV2BSDE}), and state the main result of the paper, being that the solution of an MFG with controlled diffusion coefficients provides a solution of such a McKean-Vlasov 2BSDE, see Theorem, \ref{ThMkV}. This theorem relies strongly on the representation of relaxed control problems with controlled diffusion coefficient through (classical) 2BSDEs, which proof we postpone to Section \ref{S2BSDE}. See Proposition \ref{PGirs2}. Section \ref{S6} contains the proof of Theorem \ref{ThMkV}.

\section{Formulation of the Mean Field Game}\label{SMFG}

\subsection{Notations}

A topological space $E$ will always be considered as a measurable space 
equipped with its Borel $\sigma$-field which will sometimes be denoted $\mathcal{B}(E)$. We denote by $\mathfrak{M}_+^1(E)$ and $\mathfrak{M}(E)$  the spaces of probability measures and of bounded signed measures on $(E,\mathcal{B}(E))$, respectively.  These spaces are naturally equipped with the topology of weak convergence, and the corresponding Borel $\sigma$-field.

Throughout this paper, we fix a maturity date $T>0$, positive integers $d,p_1,p_0\in\mathbbm{N}^*$, a compact Polish space $A$, and we denote $\Omega:=\mathcal{X}\times \mathcal{Q}\times \mathcal{W}\times \mathfrak{M}_+^1(\mathcal{X})$ the canonical space, where
	\begin{itemize}
		\item $\mathcal{X}:=\mathcal{C}([0,T],\mathbbm{R}^d)$ is the path space of the state process;
		\item $\mathcal{Q}$ is the set of relaxed controls, i.e. of measures $q$ on $[0,T]\times A$ such that $q(\cdot\times A)$ is equal to the Lebesgue measure. Each $q\in\mathcal{Q}$ may be identified with a measurable function $t\mapsto q_t$ from $[0,T]$ to $\mathfrak{M}_+^1(A)$ determined a.e. by $q(dt,da)=q_t(da)dt$;
		\item $\mathcal{W}:= \mathcal{W}^1\times\mathcal{W}^0$  where  $\mathcal{W}^i:=\mathcal{C}([0,T],\mathbbm{R}^{p_i}),\, i\in\{1,0\}$ 
		denote the path space of the individual noise and  that of the common noise, respectively,   
		 and we denote $\mathbbm{W}^i$  the Wiener measure on $\mathcal{W}^i$.
		
	\end{itemize}

Each of these spaces is equipped with its Borel $\sigma$ field. We also denote  $\mathcal{F}:=\mathcal{B}(\Omega)$ and $(X,Q,W,M)$  the identity (or canonical) map on $\Omega$, with $W:=(W^1,W^0)$.

On $\mathcal{X}$ (resp. $\mathcal{Q}$, $\mathcal{W}^1$, $\mathcal{W}^0$), the canonical process $X$ (resp. $Q$, $W^1$, $W^0$) generates a natural filtration $\mathbbm{F}^X$ (resp. $\mathbbm{F}^Q$, $\mathbbm{F}^{W^1}$, $\mathbbm{F}^{W^0}$). We use similar notations on product spaces.

$\mathfrak{M}_+^1(\mathcal{X})$ is equipped with a filtration $\mathbbm{F}^{M}$ defined by $\mathcal{F}^{M}_t:=\sigma(M(F):F\in\mathcal{F}^{X}_t)$. We can similarly define a filtration  $\mathbbm{F}^{X,Q,W,M}$ on  $\Omega$, which we shall rather denote $\mathbbm{F}$.

Let $\mathbbm{P}\in\mathfrak{M}_+^1(\Omega)$, $Y$ a r.v.  on $(\Omega,\mathcal{F})$ with values in a measurable space $(E,\mathcal{E})$, and $\mathcal{G}$ a sub $\sigma$-field of $\mathcal{F}$. We denote by $\mathbbm{P}\circ(Y|\mathcal{G})^{-1}$ the random measure which to some $F\in\mathcal{E}$ maps $\mathbbm{P}[Y\in F|\mathcal{G}]$.

Moreover, if $(\mathbbm{P}^{\mathcal{G}}_{\omega})_{\omega\in\Omega}$ is a regular conditional probability distribution of $\mathbbm{P}$ given $\mathcal{G}$, we have $\mathbbm{P}\circ(Y|\mathcal{G})^{-1}:(F,\omega)\longmapsto  \mathbbm{P}^{\mathcal{G}}_{\omega}(Y\in F)$, $\mathbbm{P}$ a.s.

\subsection{Controlled state process}
\label{sec:controlled}

The controlled state process is defined as a weak solution of the following relaxed SDE, whose precise meaning will be made clear in Definition \ref{def_admissible} (ii),
\begin{equation}\label{MFG}
X_t= x + \int_0^t\int_Ab_r(a,M)Q_r(da)dr + \int_0^t\int_A\sigma_r(a,M)N^{W}(da,dr).
\end{equation}
Here, $N^W:=(N^{W^1},N^{W^0})$ is a pair of orthogonal martingale measures with intensity $Q_tdt$, see e.g El Karoui \& Méléard \cite{EK_Mele}, $M:\Omega \longrightarrow \mathfrak{M}_+^1(\mathcal{X})$ is a random probability measure on $\mathcal{X}$, and
 $$
 \sigma:=(\sigma^1 \| \sigma^0), 
 ~~
 (b,\sigma^i) : [0,T]\times\mathcal{X}\times A\times \mathfrak{M}_+^1(\mathcal{X})
 \longrightarrow 
 \mathbbm{R}^d\times \mathbbm{M}_{d,p_i}(\mathbbm{R}),~~i=0,1
 $$
are  progressively measurable in the sense that for all $t\leq T$, their restriction to $[0,t]\times \mathcal{X}\times A\times \mathfrak{M}_+^1(\mathcal{X})$ is $\mathcal{B}([0,t])\otimes \mathcal{F}^X_t\otimes \mathcal{B}(A)\otimes \mathcal{F}^{M}_t$-measurable.

In order to introduce the precise meaning of \eqref{MFG}, we denote $p:=p^1+p^0$, $\bar{b}:=\left(\!\begin{array}{cc} b\\ \hline 0_p\end{array}\!\right)$, $\bar{\sigma}:=\left(\!\begin{array}{cc} \sigma\\ \hline I_p\end{array}\!\right)$, and we introduce the generator of the controlled pair $(X,W)$, defined for $(t,x,a,m)\in [0,T]\times\mathcal{X}\times A\times\mathfrak{M}_+^1(\mathcal{X})$ by: 
	$$
	\mathcal{A}^{a,x,m}_t\phi
	:=
	\bar{b}_t(x,a,m)\cdot D \phi+ \frac{1}{2}
	\bar{\sigma}\bar{\sigma}^{\intercal}_t(x,a,m):D^2\phi,
	~\mbox{for all}~
	\phi\in\mathcal{C}^2_b(\mathbbm{R}^d\times\mathbbm{R}^{p}),
	$$
where $:$ denotes the scalar product of matrices.
 
\begin{definition}\label{def_admissible}
	{\rm (i)} $\Pi^0$ denotes the set of all measures $\pi^0\in\mathfrak{M}_+^1\left(\mathcal{W}^0\times\mathfrak{M}_+^1(\mathcal{X}) \right)$ such that $W^0$ is a $(\pi^0,\mathbbm{F}^{W^0,M})$-Brownian motion.
	\\
	{\rm (ii)} For  $\pi^0\in\Pi^0$, a $\pi^0$\textbf{-admissible control} is a probability measure $\mathbbm{P}\in\mathfrak{M}_+^1(\Omega)$ with marginal $\mathbbm{P}\circ(W^0,M)^{-1}=\pi^0$, satisfying
		
		\begin{enumerate}
			\item[1] for all $\phi\in\mathcal{C}^2_b(\mathbbm{R}^d\times\mathbbm{R}^{p})$, the following process is a $(\mathbbm{P},\mathbbm{F})$-martingale:
			$$\phi(X_{t},W_{t})-\int_0^{t}\int_A \mathcal{A}^{a,X,M}_r\phi(X_r,W_r)\, Q_r(da)dr,\quad t\in[0,T];$$
			\item[2] $M$ is $\mathbbm{P}$ independent of $W^1$;
			\item[3] for all $t\in[0,T]$, $\mathcal{F}_t^{Q}$ is $\mathbbm{P}$ independent of $\mathcal{F}^{W}_T$  conditionally on $\mathcal{F}^{W}_t$, i.e. 
			\begin{equation}\label{causality}
			\mathbbm{P}[A_t\cap A_T|\mathcal{F}^W_t]=\mathbbm{P}[A_t|\mathcal{F}^W_t]\mathbbm{P}[ A_T|\mathcal{F}^W_t],
			~\mbox{for all}~
			(A_t,A_T)\in \mathcal{F}^Q_t\times\mathcal{F}^W_T.
			\end{equation}
		\end{enumerate}
We denote by  $\mathcal{P}(\pi^0)$   the set of $\pi^0$-admissible controls, and we introduce the set of \textbf{admissible controls} $\mathcal{P}(\Pi^0)$.
\end{definition}

We shall refer to \eqref{causality} as a causality condition.

\subsection{The Mean Field Game}

Let $f:[0,T]\times\mathcal{X}\times A\times \mathfrak{M}_+^1(\mathcal{X})\longrightarrow \mathbbm{R}$ be a progressively measurable map, $\xi:\mathcal{X}\longrightarrow \mathbbm{R}$ a Borel map, and define the functional 
\begin{equation}\label{J}
	J(\mathbbm{P}):=\mathbbm{E}^{\mathbbm{P}}\Big[\xi + \int_0^T\!\!\!\int_Af_r(a,M)Q_r(da)dr\Big],
	~~\mathbbm{P}\in\mathfrak{M}_+^1(\Omega).
\end{equation}

A solution of the Mean Field Game (MFG) is defined by the two following steps:
\begin{enumerate}
 \item Given the joint law $\pi^0\in\Pi^0$ of the pair $(W^0,M)$, the individual optimization problem  consists in the maximization of the functional $J$ over all weak solutions $\mathbbm{P}\in\mathcal{P}(\pi^0)$ of \eqref{MFG} in the sense of Definition \ref{def_admissible} (ii). The corresponding set of optimal solutions 
 $$
 \mathcal{P}^*(\pi^0):= \underset{\mathbbm{P}\in \mathcal{P}(\pi^0)}{\text{\rm{Argmax} }}J(\mathbbm{P}),
 ~~\mbox{for all}~~
 \pi^0\in\Pi^0,
 $$
defines a correspondence $\mathcal{P}^*$ from $\Pi^0$ to $\mathcal{P}(\Pi^0)$.
\item A strong solution of the MFG is an optimal probability $\mathbbm{P}^*\in\mathcal{P}^*(\pi^0)$ such that $M = \mathbbm{P}^*\circ(X|\mathcal{F}^{W^0})^{-1} $ a.s., i.e. under $\mathbbm{P}^*$, $M$ is the conditional law of the state process $X$ given the common noise $W^0$. 
\end{enumerate}
For technical reason explained below, we need to consider the following weaker notion.

\begin{definition}{\rm (Carmona, Delarue \& Lacker \cite{LackerCommonNoise})}\label{weaksol}
	A weak relaxed  solution of the MFG is a probability  $\mathbbm{P}\in\mathfrak{M}_+^1(\Omega)$ such that:
	
	\vspace{3mm}
	\quad\rm{\bf{Individual optimality:}}\quad $\mathbbm{P}\in\mathcal{P}^*(\pi^0)$, for some $\pi^0\in\Pi^0$ ;
	
	\vspace{3mm}
	\quad\rm{\bf{Weak Equilibrium:}}\quad  $M =\mathbbm{P}\circ(X|\mathcal{F}^{M,W^0})^{-1}$, $\mathbbm{P}$ a.s.
	
\end{definition}

Observe that the weak equilibrium condition  in the last definition is 
 indeed weaker than the strong equilibrium requirement $M = \mathbbm{P}\circ(X|\mathcal{F}^{W^0})^{-1} $ a.s. which is thus named strong solution of the MFG by Carmona \& Delarue \cite{CarmonaDelarueII}. The reason for introducing this weak notion of solution in \cite{LackerCommonNoise} is recalled in Remark \ref{rem:conditionalexp} below.

%

\section{Weak relaxed Nash equilibrium}\label{MFGSol}

\subsection{Assumptions and main results}\label{S3.1}

The following assumption will be needed to prove existence of weak relaxed solutions of the MFG.
\begin{assumption}\label{H_existence}\
	\begin{enumerate}
		\item[\rm (i)] $b,\sigma,f$ are bounded and continuous in $(x,a,m)$, for all $t$, and $\xi$ is bounded continuous;
		\item[\rm (ii)] for every probability measure $\mathbbm{Q}$ on $\mathcal{Q}\times \mathcal{W}\times \mathfrak{M}_+^1(\mathcal{X})$ under which $W$ is a Brownian motion, there exists a unique $\mathbbm{P}\in\mathfrak{M}_+^1(\Omega)$ with marginal $\mathbbm{P}\circ (Q,W,M)^{-1}= \mathbbm{Q}$ and satisfying Item 1 of Definition \ref{def_admissible} (ii).
	\end{enumerate}
\end{assumption}

Assumption \ref{H_existence} (ii) is an existence and uniqueness condition for the SDE \eqref{MFG}. It is verified for instance when $b,\sigma$ are bounded and locally Lipschitz in $x$, uniformly in $(t,a,m)$. This can be seen by considering the strong solution of the controlled SDE, which is then driven by martingale measures, see \cite{EK_Mele} for basic results concerning such SDEs.

We may now state the main result of this section.

\begin{theorem}\label{T_existence}
Under Assumption \ref{H_existence}, there exits at least one weak relaxed solution of the MFG in the sense of Definition \ref{weaksol}.
\end{theorem}

The proof of this theorem will mainly rely on the  Kakutani-Fan-Glicksberg fixed point Theorem. The Appendix Section of the present paper provides an introduction to set valued functions (or correspondences) which will be used extensively in this paper, we refer to \cite{aliprantis} Chapter 17.


\subsection{Preliminary topological results}\label{S3.2}

The aim of this subsection is to prove the following topological results.

\begin{proposition}\label{PropReg}

\noindent {\rm (i)} $\Pi^0$  is a closed convex subset of $\mathfrak{M}_+^1\left(\mathcal{W}^0\times\mathfrak{M}_+^1(\mathcal{X})\right)$, and consequently of $\mathfrak{M}\left(\mathcal{W}^0\times\mathfrak{M}_+^1(\mathcal{X})\right)$;

\noindent {\rm (ii)} $\mathcal{P}$ is a continuous correspondence with nonempty compact convex values;

\noindent {\rm (iii)} $\mathcal{P}^*$ is an upper hemicontinuous correspondence with nonempty compact convex values, moreover, $\mathcal{P}^*(\Pi^0)$ is closed.
\end{proposition}

\begin{remark}\label{rem:conditionalexp}{\rm
Recall that a strong solution of the MFG is a probability measure $\mathbbm{P}^*\in \mathcal{P}^{*}(\pi^*)$, for some $\pi^*\in\Pi^0$, such that  $M= \mathbbm{P}^*\circ(X|\mathcal{F}^{W^0})^{-1}$ $\mathbbm{P}^*$ a.s., or equivalently 
\begin{eqnarray}\label{FixedPoint}
	\pi^* \in
	\Phi\circ\mathcal{P}^{*}(\pi^*)
	&\mbox{where}&
	\Phi(\mathbbm{P})
	:=
	\mathbbm{W}^0\circ\Big(W^0,\, \mathbbm{P}\circ\big(X|\mathcal{F}^{W^0}\big)^{-1}\Big)^{-1}
\end{eqnarray}
If the map $\Phi$ were continuous, then we may conclude from Proposition \ref{PropReg} that such a fixed point exists, by the Kakutani fixed point theorem, see Theorem \ref{Kakutani}. Unfortunately, the conditional expectation operator is not continuous, in general. For this reason, the proof strategy used in \cite{LackerCommonNoise} consists in introducing a discretization of the common noise $W^0$, so as to reduce the fixed point problem to the context of a finite $\sigma$-field where the conditional expectation is indeed continuous. The weak solution of the MFG is then obtained as a limiting point of the solutions of the MFG problems with finite approximation of the common noise. See Section \ref{sec:discretenoise} below.
}
\end{remark}

\begin{prooff}{}
{\bf \hspace{-5mm} of Proposition \ref{PropReg} (i)}
	By the Lévy characterization, $W^0$ is an $\mathbbm{F}^{W^0,M}$-Brownian motion iff $W^0$ and $W^0_t(W^0_t)^{\intercal}-tId_{p_0}$ are martingales. As  the set of solutions of a Martingale problem is convex, see Corollary 11.10 in \cite{jacod79}, we immediately deduce that $\Pi^0$ is convex.

	We now show that $\Pi^0$ is closed.  Assume that a sequence $(\pi^0_n)_{n\in\mathbbm{N}}$ of elements of $\Pi^0$ converges weakly to some $\pi^0$. By the Lévy criterion, we have for all $s\leq t\in[0,T]$ and all bounded continuous $\mathcal{F}^{W^0,M}_s$-measurable $\phi_s$,
	\begin{equation}\label{EW0}
	\mathbbm{E}^{\pi^0_n}[(W^0_t-W^0_s)\phi_s]=0  \text{ and }
	\mathbbm{E}^{\pi^0_n}[(W^0_t(W^0_t)^{\intercal}-W^0_s(W^0_s)^{\intercal} -(t-s)Id_m)\phi_s]=0.
	\end{equation}
As $(W^0_t-W^0_s)\phi_s$ and $(W^0_t(W^0_t)^{\intercal}-W^0_s(W^0_s)^{\intercal} -(t-s)Id_m)\phi_s$ are continuous uniformly integrable r.v. under $(\pi^0_n)_n$, we may send $n$ to infinity in \eqref{EW0} and obtain that $W^0$ is a $(\pi^0,\mathbbm{F}^{W^0,M})$-Brownian motion, see \cite{billingsley86} Theorem 3.5. 
\end{prooff}	

\begin{prooff}{}
{\bf \hspace{-5mm} of Proposition \ref{PropReg} (iii)}
We now show that (iii) is a consequence of (ii), whose proof is postponed. As $f,\xi$ are bounded continuous, the map $J$ introduced in \eqref{J}  is continuous on $\mathfrak{M}^1_+(\Omega)$. Then, since  $\mathcal{P}$ is continuous with nonempty compact values, it follows directly by Theorem \ref{Berge}  that $\mathcal{P}^*$ is upper hemicontinuous and takes nonempty compact values. 

We next show that it takes convex values. Let  $\pi^0\in \Pi^0$, $\mathbbm{P}^1,\mathbbm{P}^2$ be elements of $\mathcal{P}^{\star}(\pi^0)$ i.e maximizers of $\mathbbm{E}^{\mathbbm{P}}[J]$ within $\mathcal{P}(\pi^0)$ and let $\alpha\in[0,1]$. Since $\mathcal{P}$ takes convex values then $\alpha\mathbbm{P}^1+(1-\alpha)\mathbbm{P}^2\in\mathcal{P}(\pi^0)$ and since $\mathbbm{E}^{\mathbbm{P}^1}[J]
=\mathbbm{E}^{\mathbbm{P}^2}[J]= 	\max_{\mathbbm{P}\in \mathcal{P}(\pi^0)}\mathbbm{E}^{\mathbbm{P}}[J]$, it follows that $\mathbbm{E}^{\alpha\mathbbm{P}^1+(1-\alpha)\mathbbm{P}^2}[J]= \max_{\mathbbm{P}\in \mathcal{P}(\pi^0)}$. Hence  $\alpha\mathbbm{P}^1+(1-\alpha)\mathbbm{P}^2$ also is a maximizer of $\mathbbm{E}^{\mathbbm{P}}[J]$ within $\mathcal{P}(\pi^0)$ and therefore belongs to $\mathcal{P}^*(\pi^0)$.

It remains to prove that $\mathcal{P}^*(\Pi^0)$ is closed. Since $\mathcal{P}^*$ is uhc and compact valued, then it has a closed graph, see Proposition \ref{alipran} Item 1. Now let $\mathbbm{P}^n\longrightarrow\mathbbm{P}$ with $\mathbbm{P}^n\in\mathcal{P}^*(\Pi^0)$ for all $n$. By construction of $\mathcal{P}^*$, we have that for all $n$, $\mathbbm{P}^n\in\mathcal{P}^*(\mathbbm{P}^n\circ(M,W^0)^{-1})$, and by continuity of marginals, that $\mathbbm{P}^n\circ(M,W^0)^{-1}$ tends to $\mathbbm{P}\circ(M,W^0)^{-1}$ which belongs to $\Pi^0$ by the closedness property established in (i) of the present proof. So by the closed graph property, $\mathbbm{P}\in\mathcal{P}^*(\mathbbm{P}\circ(M,W^0)^{-1})\subset\mathcal{P}^*(\Pi^0)$
and the proof is complete.
\end{prooff}
The rest of this section is dedicated to the proof of Proposition \ref{PropReg} (ii). We start with an immediate consequence of Proposition \ref{PropReg} (i).
\begin{corollary}\label{L_K_W_mu}
Let $\Pi:=\{\pi:=\mathbbm{W}^1\otimes\pi^0:~\pi^0\in\Pi^0 \}$. Then,
\\
{\rm (i)} $\Pi$ is a closed convex subset of $\mathfrak{M}_+^1(\mathcal{W}\times\mathfrak{M}_+^1(\mathcal{X}))$;
\\
{\rm (ii)} the map $\mathbf{T}:
		\pi^0\in\Pi^0\longmapsto\pi:=\mathbbm{W}^1\otimes\pi^0\in \Pi$ is a homeomorphism;
\\
{\rm (iii)} if $\mathcal{K}^0$ is a compact (resp. convex) subset of $\Pi^0$, then $\mathcal{K}:=\mathbf{T}(\mathcal{K}^0)$ is a compact (resp. convex) subset of $\Pi$.
\end{corollary}

We next consider a further extension of the probability measures $\pi\in\Pi$:
 $$
 \mathfrak{Q}_c(\pi):=\big\{\mathbbm{Q}\in\mathfrak{M}_+^1\big(\mathcal{Q}\times \mathcal{W}\times \mathfrak{M}_+^1(\mathcal{X})\big): \mathbbm{Q}\circ(W,M)^{-1}= \pi~\mbox{and}~ \mathbbm{Q}~\mbox{satisfies}~\eqref{causality}\big\},
 $$
where the subscript ``c'' stands for the causality condition \eqref{causality}.

\begin{lemma}\label{L2}
{\rm (i)} $\mathfrak{Q}_c(\Pi)$ is closed convex;
\\
{\rm (ii)} Let $\mathcal{K}^0$ be a compact (resp. convex) subset  of $\Pi^0$, and set $\mathcal{K}:=\mathbf{T}(\mathcal{K}^0)$; then $\mathfrak{Q}_c(\mathcal{K})$ is a compact (resp. convex) subset of $\mathfrak{Q}_c(\Pi)$; 
\\
{\rm (iii)} the correspondence $\mathfrak{Q}_c:\pi\in\Pi\longmapsto\mathfrak{Q}_c(\pi)$ is continuous.
\end{lemma}

\begin{proof}
Throughout this proof, we denote $\mathfrak{Q}:=\big\{\mathbbm{Q}\in\mathfrak{M}_+^1\big(\mathcal{Q}\times \mathcal{W}\times \mathfrak{M}_+^1(\mathcal{X})\big): \mathbbm{Q}\circ(W,M)^{-1}\in \Pi~\mbox{and}~ \mathbbm{Q}~\mbox{satisfies}~\eqref{causality}\big\}=\mathfrak{Q}_c(\Pi)$.
\\
(i) Since $\Pi$ is itself convex and closed by Corollary \ref{L_K_W_mu}, then the first item above is stable by convergence or convex combinations. By Theorem 3.11 in \cite{LackerCompatibility}, since $W$ has independent increments (with respect to its own filtration), the second item above holds iff  for all $t\leq s$, $W_t-W_s$ is $\mathbbm{Q}$-independent of $\mathcal{F}^{Q,W}_s$. This condition is also stable under convergence or convex combinations, so $\mathfrak{Q}$ is closed and convex. 

(ii) Closeness and convexity of $\mathfrak{Q}_c(\mathcal{K})$ follow from the same arguments as above.
Its tightness (hence relative compactness) follows from the compactness of $\mathcal{Q}$ and the tightness of $\{\mathbbm{Q}\circ(W,M)^{-1}:\,\mathbbm{Q}\in\mathfrak{Q}_c(\mathcal{K})\}=\mathcal{K}$.

(iii) We decompose $\mathfrak{Q}_c$ as the composition of two continuous correspondences $\Gamma_1,\Gamma_2$ which we now introduce.
We denote $\mathcal{K}':=\{\mathbbm{Q}\circ(Q,W)^{-1}:\mathbbm{Q}\in\mathfrak{Q}\}$ i.e. the set of laws in $\mathfrak{M}_+^1(\mathcal{Q}\times\mathcal{W})$ for which $W$ is an $\mathbbm{F}^{Q,W}$-Brownian motion. With arguments similar to what we have seen for $\Pi^0$ or $\Pi$, it is easy to see that $\mathcal{K}'$ is closed convex, and is even compact thanks to the compactness of $\mathcal{Q}$.

We define  the correspondence $\Gamma_1$ which to any $\pi\in\Pi$ maps the subset $\{\pi\}\times \mathcal{K}'$ of $\Pi\times \mathcal{K}'$.
We also define $ \Gamma_2$ which to any $(\pi,\pi')$ in $\Pi\times \mathcal{K}'$ maps the set 
 $$
 \left\{\mathbbm{Q}\in\mathfrak{Q}: \mathbbm{Q}\circ(Q,W)^{-1}=\pi',\quad
                                                                  \mathbbm{Q}\circ(W,M)^{-1}=\pi\right\}.
 $$
It is clear that  $\mathfrak{Q}_c=\Gamma_2\circ\Gamma_1$.

$\Gamma_1$ is the product of the continuous function $\pi\longmapsto\pi$ and of the  correspondence $\pi\longmapsto \mathcal{K}'$ which is compact valued and constant hence continuous, apply Proposition \ref{PropReg} Item 3 for instance. So $\Gamma_1$ is continuous as the product of continuous compact valued correspondences, see Theorem 17.28 in \cite{aliprantis}.

$\Gamma_2$ is the restriction on $\Pi\times \mathcal{K}'$ of the inverse $\psi^{-1}$ of the mapping 
$\psi:\mathbbm{Q}\longmapsto (\mathbbm{Q}\circ(Q,W)^{-1},\mathbbm{Q}\circ(W,M)^{-1})$. Adapting Theorem 3 in \cite{eifler} for example, we have that $\psi$ is an open mapping. Then, by Theorem 17.7 in \cite{aliprantis},  $\psi^{-1}$ (or its restriction $\Gamma_2$) is lower hemicontinuous. It is immediate that $\Gamma_2$ has a closed graph, however, its range is not compact so we can not conclude immediately that it is uhc.

Let us fix some compact subset $\mathcal{K}^0$ of $\Pi^0$ and $\Gamma_2^{\mathcal{K}^0}$ the restriction of $\Gamma_2$ on $\mathcal{K}^0$. Then $\Gamma_2^{\mathcal{K}^0}$ is still lhc with closed graph but this time has compact range hence is uhc by the closed graph theorem, see Proposition \ref{PropReg} (ii). It is therefore continuous. $\Gamma_2$ is compact valued hence can also be seen as a function with values in the metric space of compact subsets of $\mathfrak{M}_+^1(\mathcal{Q}\times\mathcal{W}\times\mathfrak{M}_+^1(\mathcal{X}))$, equipped with the Hausdorff metric. By Proposition \ref{PropReg} (iii), $\Gamma_2$ is continuous on a certain set as a correspondence, iff it is continuous as a function for the Hausdorff metric. What we have seen is that $\Gamma_2$ is in fact continuous on every compact subset of $\Pi\times \mathcal{K}'$, and in a metric space, a function which is continuous on every compact set is continuous everywhere. So $\Gamma_2$ is continuous everywhere, hence $\mathfrak{Q}_c$ is continuous as the composition of continuous correspondences, see Proposition \ref{alipran} Item 4.
\end{proof}

We finally lift the set $\mathfrak{Q}_c(\Pi)$ by the map
 $$
 \mathbbm{Q}\in\mathfrak{Q}_c(\Pi)
 \longmapsto
 \Psi(\mathbbm{Q}) := \mathbbm{P}\in\mathcal{P}(\Pi^0)
 ~\mbox{if and only if}~
 \mathbbm{P}\circ (Q,W,M)^{-1}=\mathbbm{Q},
 $$
where the existence and uniqueness of $\mathbbm{P}$ is guaranteed by Assumption \ref{H_existence}.

\begin{lemma}\label{L3}
{\rm (i)} $\mathcal{P}(\Pi^0)$ is a closed convex subset of $\mathfrak{M}_+^1(\Omega)$, and $\mathcal{P}(\mathcal{K}^0)$ is compact (resp. convex) for all compact (resp. convex) subset $\mathcal{K}^0$ of $\Pi^0$;
\\
{\rm (ii)} $\Psi$ is a homeormorphism from $\mathfrak{Q}_c(\Pi)$ to $\mathcal{P}(\Pi^0)$.
\end{lemma}

\begin{proof}
(i) 	By definition, $\mathbbm{P}\in\mathfrak{M}_+^1(\Omega)$ belongs to $\mathcal{P}(\Pi^0)$ iff
	\begin{enumerate}
		\item[a.] $\mathbbm{P}\circ(Q,W,M)^{-1}$ belongs to $\mathfrak{Q}_c(\Pi)$
		\item[b.] for all $\phi\in\mathcal{C}^2_b(\mathbbm{R}^d\times\mathbbm{R}^{p})$, 
		$$\phi(X_{t},W_{t})-\int_0^{t}\int_A \mathcal{A}^{a,X,M}_r\phi(X_r,W_r)\, Q_r(da)dr,\quad t\in[0,T]$$ is a $(\mathbbm{P},\mathbbm{F})$-martingale.
	\end{enumerate} 
As $\mathfrak{Q}_c(\Pi)$ is convex and closed by Lemma \ref{L2}, it is clear that the set of $\mathbbm{P}$ verifying Item 1 above is convex and closed.
Then, since the set of solutions of a martingale problem is convex (see Corollary 11.10 in \cite{jacod79}), and since the coefficients $b,\sigma$ are bounded and continuous in $(x,a,m)$ for fixed $t$, the set of probability measures verifying Item 2 above is also closed convex. This shows that $\mathcal{P}(\Pi^0)$ is closed convex. 

We next prove the second part of (i). We fix some compact convex subset $\mathcal{K}^0$ of $\Pi^0$. It is immediate by construction that $\mathcal{P}(\mathcal{K}^0)$ remains closed convex, so we are left to prove that it is relatively compact.
By boundedness of $b,\sigma$, the set $\{\mathbbm{P}\circ X^{-1}:\mathbbm{P}\in\mathcal{P}(\mathcal{K}^0)\}$ is tight (see Theorem 1.4.6 in \cite{stroock} for instance), and by compactness of $\mathcal{Q}$ and tightness of $\mathbbm{W}^1\otimes\mathcal{K}^0$  we  have  that $\{\mathbbm{P}\circ (Q,W,M)^{-1}:\mathbbm{P}\in\mathcal{P}(\mathcal{K}^0)\}$ is tight. So  $\mathcal{P}(\mathcal{K}^0)$ is tight and therefore relatively compact which concludes the proof.
\\
(ii) It is clear that $\Psi$ is a bijection, an that its reciprocal $\Psi^{-1}$ (defined by $\Psi^{-1}(\mathbbm{P})=\mathbbm{P}\circ (Q,W,M)^{-1}$) is continuous. 
	
	Let $\mathbbm{P}_n\longrightarrow \mathbbm{P}$ in $\mathfrak{Q}_c(\Pi)$ then we also have $\mathbbm{P}_n\circ(M,W^0)^{-1}\longrightarrow \mathbbm{P}\circ(M,W^0)^{-1}$ so the measures $(\mathbbm{P}_n\circ(M,W^0)^{-1})_n$ and $\mathbbm{P}\circ(M,W^0)^{-1}$ belong to some compact subset $\mathcal{K}^0$ of $\Pi^0$ and the measures $(\mathbbm{P}_n)_n$ and $\mathbbm{P}$ belong to $\mathfrak{Q}_c(\mathcal{K})$ where $\mathcal{K}:=\mathbbm{W}^1\otimes\mathcal{K}^0$. So it is enough to show that $\Psi$ is continuous on $\mathfrak{Q}_c(\mathcal{K})$ for any compact subset $\mathcal{K}^0$ of $\Pi^0$.
	
	We fix $\mathcal{K}^0$ and $\mathcal{K}:=\mathbbm{W}^1\otimes\mathcal{K}^0$. By construction, the restriction of $\Psi$ induces a bijection $\Psi_{\mathcal{K}^0}$ from  $\mathfrak{Q}_c(\mathcal{K})$ onto $\mathcal{P}(\mathcal{K}^0)$ which are both compact, see Lemma \ref{L2} and the first part (i) of the present lemma.
	$\Psi_{\mathcal{K}^0}^{-1}$ is the marginal mapping $\mathbbm{P}\mapsto\mathbbm{P}\circ(Q,W,M)^{-1}$ restricted on $\mathcal{P}(\mathcal{K}^0)$ hence is continuous. So $\Psi_{\mathcal{K}^0}^{-1}$ is a continuous bijection between compact sets, hence a homeomorphism. $\Psi_{\mathcal{K}^0}^{-1}$ is therefore continuous, meaning that $\Psi$ is continuous on $\mathfrak{Q}_c(\mathcal{K})$ and the proof is complete.
\end{proof}

We can now conclude the proof of  Proposition \ref{PropReg}.
\begin{prooff}{} {\bf \hspace{-5mm} of Proposition \ref{PropReg} (ii)}
	$\mathcal{P}$ may now be written as the composition $\Psi\circ\mathfrak{Q}_c\circ \mathbf{T}$ where $\mathbf{T}$, $\mathfrak{Q}_c$ and $\Psi$ were respectively introduced in Lemmas \ref{L_K_W_mu}, \ref{L2} and \ref{L3}. So thanks to these three lemmas, $\mathcal{P}$ is a continuous correspondence as the composition of two continuous functions and a continuous correspondence, see Proposition \ref{alipran} Item 4. 
	
	For every $\pi^0\in\Pi^0$, we have that $\mathcal{P}(\pi^0)$ is compact convex by Lemma \ref{L3} (i).
	Finally, $\mathcal{P}$ takes non-empty values thanks to Assumption \ref{H_existence} Item 1.
\end{prooff}

\subsection{Discretized strong equilibria}\label{sec:discretenoise}

This section follows the proof strategy of \cite{LackerCommonNoise} as commented earlier in Remark \ref{rem:conditionalexp}. The main novelty in what follows is our reformulation of the problem given in \eqref{FixedPoint}. Under this perspective, all our analysis is made on the space $\mathfrak{M}_+^1(\Omega)$. We believe that this point of view simplifies some technical issues, and is the key ingredient for allowing the control in the diffusion coefficient.

\begin{notation}\label{Not_discret}
For each $n \geq  1$, let $t^n_i := i2^{-n}T$ for $i = 0, . . . , 2^n$. For every $n$, we fix a partition $c_n := \{C_1^n , \cdots , C_n^n\}$ of $\mathbbm{R}^{p_0}$ into $n$ Borel sets of
strictly positive Lebesgue measure, such that for all $n$, $c_{n+1}$ is a refinement of $c_n$, and $\mathcal{B}(\mathbbm{R}^{p_0}) = \sigma\left(\bigcup_n c_n\right)$. 
For a given $n$, and
$I = (i_1 ,\cdots , i_{2^n} ) \in \{1,\cdots , n\}^{2^n}, k \leq 2^n$, we  define $S_I^{n,k}$ as the set of paths with increments up until time $k$
 in $C_{i_1}^n,\cdots ,C_{i_k}^n$ i.e.
$$S_I^{n,k} := \{\omega^0\in\mathcal{W}^0: \omega^0_{t_j^n} - \omega^0_{t_{j-1}^n} \in C_{i_j}^n , \text{for all } j = 1,\cdots, k\}.$$ We also denote $S_I^{n}:= S_I^{n,2^n}$.
The $S_I^{n}$’s, $I \in \{1,\cdots, n\}^{2^n}$, form a finite partition of $\mathcal{W}^0$ , each $S_I^{n}$  having a strictly
positive $\mathbbm{W}^0$-measure. 

For all $n$ we denote $\mathcal{F}^{n,W^0}:= \sigma(S_I^{n}: I \in \{1,\cdots, n\}^{2^n})$ and 
for all $t\in[0,T]$, we denote $\mathcal{F}^{n,W^0}_t:= \sigma(S_I^{n,j}: I \in \{1,\cdots, n\}^{2^n}, j\leq k_t)$ where $k_t$ is the largest integer such that $t_{k_t}\leq t$. 

Finally, for all $n$, we introduce the mapping
$\hat{X}^n:
\mathcal{X}\longrightarrow\mathcal{X}$
such that for all $k< 2^n$ and $t\in[t_k^n,t_{k+1}^n[$, $\hat{X}^n_t= \frac{2^n}{T}(t-t_k)X_{t_k}+\frac{2^n}{T}(t_{k+1}-t)X_{t_{k-1}}$.

\end{notation}

The following facts may be found in \cite{LackerCommonNoise} Subsection 2.4.2 and the proof of its Lemma 3.6 (second step).
\begin{remark}\label{R_discret}
{\rm (i)} For all $t\in[0,T]$, $\mathcal{F}^{W^0}_t=\sigma\big(\cup_n \mathcal{F}^{n,W^0}_t\big)$;
\\
{\rm (ii)} $(\mathcal{F}^{n,W^0}_t)_{t\geq 0}$ is a sub filtration of $\mathbbm{F}^{W^0}$;
\\
{\rm (iii)} for all $n$, $\hat{X}^n$ is continuous,  and $\hat{X}^n\longrightarrow X$ as $n\to\infty$ uniformly on the compact sets of $\mathcal{X}$.
\end{remark}

\begin{definition}\label{D_discret}
A \textbf{discretized strong Nash equilibrium of order }$n$, is a probability measure $\mathbbm{P}\in\mathcal{P}^*(\Pi^0)$ such that 
\begin{equation}\label{E_discret_0}
M = \mathbbm{P}\circ(\hat{X}^n|\mathcal{F}^{n,W^0})^{-1}\quad \mathbbm{P}\text{  a.s.}
\end{equation}	
	
\end{definition}

\begin{proposition}\label{P_discret}
	For every $n$, there exists a discretized strong Nash equilibrium of order $n$.
\end{proposition}

We will prove this first existence result by means of the Kakutani fixed point theorem, thanks to the regularity of the correspondence $\mathcal{P}^*$. However, such a fixed point theorem holds in a compact convex set, and our set $\Pi^0$ is not compact, so we now construct a smaller (and compact) set, in which we will apply that theorem.

\begin{notation}\label{N_K_X}
	If $\mathbbm{P}\in \mathfrak{M}_+^1(\mathcal{X})$ is such that $X$ is a $\mathbbm{P}$-semimartingale, we denote by $A^{\mathbbm{P}}$ and $M^{\mathbbm{P}}$ the bounded variation and the martingale components of $X$ under $\mathbbm{P}$.
	
	$\mathcal{K}^{X}$ denotes the closure of the space of elements of $\mathfrak{M}_+^1(\mathcal{X})$ under which $X$ is a  semimartingale for which $ |A^{i,\mathbbm{P}}|,\, i\leq d$ and $Tr(\langle M^{\mathbbm{P}}\rangle)$ are absolutely continuous with derivatives bounded by
	$C$ $dt\otimes d\mathbbm{P}$ a.e., where $C$ is a fixed constant bounding $b$ and $\bar{\sigma}\bar{\sigma}^{\intercal}$ for the sup norm.
	
\end{notation}

\begin{lemma}
	$\mathcal{K}^{X}$ is a compact subset of $\mathfrak{M}_+^1(\mathcal{X})$. 
\end{lemma}
\begin{proof}
	It is well known that any family of laws of continuous diffusions with bounded coefficients is tight (see \cite{stroock} Theorem 1.4.6 for instance) so $\mathcal{K}^{X}$ is the closure of a tight set, hence of a relatively compact set by the Prohorov's theorem.

\end{proof}

For all $n\in\mathbbm{N}^*$, we also denote 
 $$
 \mathcal{K}^{X}_n:=\{\mathbbm{P}\circ(\hat{X}^n)^{-1}:\mathbbm{P}\in\mathcal{K}^{X}\}.
 $$
By the tightness of $\mathcal{K}^{X}$ we may introduce an increasing sequence of compact subsets $(K^{\infty}_k)_{k\in\mathbbm{N}^*}$ of $\mathcal{X}$ such that 
 $$
 \mathbbm{P}[X\in K^{\infty}_k]\geq 1-\frac{1}{k}
 ~~\mbox{for all}~~
 k>0~~
 \mbox{and}~~
 \mathbbm{P}\in\mathcal{K}^{X}.
 $$
Finally, we denote
 $$
 K_k^n:=\hat{X}^n(K^{\infty}_k)
 ~~\mbox{and}~~
 \bar{K}_k:= \underset{n\in\mathbbm{N}\cup\{\infty\} }{\bigcup}K_k^n,
 ~~\mbox{for all}~~
 k,n\in\mathbb{N}.
 $$
\begin{lemma}
For all $k,n$, $K_k^n$ and $\bar{K}_k$ are compact, and $\mathcal{K}^{X}_n$ is tight.
\end{lemma}
\begin{proof}
Compactness of $K_k^n$ follows from the continuity of $\hat{X}^n$ which therefore maps compact sets onto compact sets.
	
We next prove that $\mathcal{K}^{X}_n$ is tight. Let $\mathbbm{Q}=\mathbbm{P}\circ(\hat{X}^n)^{-1}\in\mathcal{K}^{X}_n$, for some $\mathbbm{P}\in\mathcal{K}^{X}$. Then, for all $k$, we have $\mathbbm{Q}[K_k^n]=\mathbbm{P}[\hat{X}^n\in K_k^n]\geq \mathbbm{P}[X\in K^{\infty}_k]\geq 1-\frac{1}{k}$. Since this holds for any $\mathbbm{Q}\in\mathcal{K}^{X}_n$ then the announced tightness is shown.
	
It remains to prove that $\bar{K}_k$ is compact. Fix a sequence $(x_n)_{n\geq 0}$ in $\bar{K}_k$. Either there exists some $(i_1,\cdots,i_N)\in\bar{\mathbbm{N}}^N$ such that $(x_n)_{n\geq 0}$ remains in the compact set $\bigcup_{j\leq N}K^{i_j}_k$, in which case that sequence admits a converging subsequence, or we can assume (up to an extraction which we omit) that there exists a strictly increasing sequence $(p_n)_n$ such that for all $n$, $x_n\in K_k^{p_n}$.
	
	Then for all $n$ we may consider some $y_n\in K^{\infty}_k$ such that $x_n=\hat{X}^n(y_n)$, and since $K^{\infty}_k$ is compact, we may assume (again up to the extraction of a subsequence) that $y_n$ converges to some $y$ in $K^{\infty}_k$. We now conclude the proof by showing that $x_n$ also tends to $y$, hence that any sequence of $\bar{K}_k$ admits a converging subsequence in $\bar{K}_k$.
	Indeed we have $$|x_n-y|= |\hat{X}^{p_n}(y_n)-y|\leq |\hat{X}^{p_n}(y_n)-y_n|+|y_n-y|.$$
	The second term on the right hand side tends to zero, and since $p_n$ is strictly increasing, then $\hat{X}^{p_n}$ tends uniformly to $X$ on compact sets, and in particular on  $K^{\infty}_k$ (see Remark \ref{R_discret}, Item 3) so $|\hat{X}^{p_n}(y_n)-y_n|$ tends to zero and the proof is complete.
\end{proof}

We now introduce the set in which we will find the discretized equilibriums: 
 \begin{equation}\label{E_discret_set}
 \Pi^0_c
 :=
 \Big\{ \pi^0\in\Pi^0: \pi^0(\bar{K}_k)\geq 1-\frac{1}{k}
                                ~\mbox{for all}~k>0
 \Big\}.
\end{equation}

\begin{lemma}\label{L_discret_set}
For all $n$, $\Pi^0_c$ is a compact convex set.
\end{lemma}
\begin{proof}
We fix $n$.
It is immediate by construction that $\Pi^0_c$ is tight hence relatively compact. Moreover, $\Pi^0$ is convex (see Proposition \ref{PropReg} Item 1) and \eqref{E_discret_set} is stable by convex combination, so $\Pi^0_c$ is also convex.

We proceed showing that $\Pi^0_c$ is closed. Since $\Pi^0$ is closed (see Proposition \ref{PropReg} Item 1), it is enough to show that \eqref{E_discret_set} is stable under convergence. We fix a converging sequence $\pi^j\longrightarrow \pi$ were $\pi^j\in\Pi^0_c$ for all $j$.

By the Skorohod representation theorem (see \cite{billingsley86} Theorem 6.7 for instance), there exists a probability space $(\tilde{\Omega},\tilde{\mathcal{F}},\tilde{\mathbbm{P}})$ on which there exist random measures $M^j$ of law $\pi^j\circ M^{-1}$ and $M^{lim}$ of law $\pi\circ M^{-1}$, and a $\tilde{\mathbbm{P}}$-null set $\mathcal{N}$ such that for all $\omega$ in $\mathcal{N}^c$, $M^j(\omega)\rightarrow M^{lim}(\omega)$ weakly.
	Since the sets $\bar{K}_k$ are closed, a consequence of Portemanteau's theorem (see \cite{billingsley86} Theorem 2.1 for instance), is that for all $k$ and $\omega\in\mathcal{N}^c$,
	\begin{equation}\label{E_K_input_2}
		M^{lim}(\omega)(\bar{K}_k)\geq \underset{j}{\text{limsup }}M^j(\omega)(\bar{K}_k).
	\end{equation}
	Then, taking the expectation in \eqref{E_K_input_2} and by the reversed Fatou's lemma, we get that for all $k$,
	\begin{equation}
		\mathbbm{E}^{\tilde{\mathbbm{P}}}[M^{lim}(\bar{K}_k)]\geq \mathbbm{E}^{\tilde{\mathbbm{P}}}[\underset{j}{\text{limsup }}M^j(\bar{K}_k)]\geq \underset{j}{\text{limsup }}\mathbbm{E}^{\tilde{\mathbbm{P}}}[M^j(\bar{K}_k)],
	\end{equation}
	hence that $\mathbbm{E}^{\pi}[M(\bar{K}_k)] \geq \underset{j}{\text{limsup }}\mathbbm{E}^{\pi_j}[M(\bar{K}_k)]\geq 1-\frac{1}{k}$. So \eqref{E_discret_set} holds under $\pi$ and the proof is complete.
\end{proof}

We may now prove the main result of this subsection.
\begin{prooff}{} {\bf \hspace{-5mm} of Proposition \ref{P_discret}}\quad
	We first note that $M = \mathbbm{P}\circ(\hat{X}^n|\mathcal{F}^{n,W^0})^{-1}$, $\mathbbm{P}$ a.s. is equivalent to having 
	$\mathbbm{P}\circ(W^0,M)^{-1}= \mathbbm{P}\circ(W^0,\mathbbm{P}\circ(\hat{X}^n|\mathcal{F}^{n,W^0})^{-1})^{-1}
	= \mathbbm{W}^0\circ(W^0, \mathbbm{P}\circ(\hat{X}^n|\mathcal{F}^{n,W^0})^{-1})^{-1}$.
	 
	 We introduce on $\mathcal{P}(\Pi^0)$ the mapping 
	 $$\Phi_n:\mathbbm{P}\longmapsto \mathbbm{W}^0\circ\left(W^0, \mathbbm{P}\circ(\hat{X}^n|\mathcal{F}^{n,W^0})^{-1}\right)^{-1},$$ and show that it is continuous on that set.
	 
	 We fix a converging sequence $\mathbbm{P}^k\longrightarrow\mathbbm{P}$ in $\mathcal{P}(\Pi^0)$. By Theorem 4.11 in \cite{kallenberg}, in oder to show that $\mathbbm{W}^0\circ\big(W^0, \mathbbm{P}^k(\hat{X}^n|\mathcal{F}^{n,W^0})^{-1}\big)^{-1}\longrightarrow\mathbbm{W}^0\circ\big(W^0, \mathbbm{P}^k(\hat{X}^n|\mathcal{F}^{n,W^0})^{-1}\big)^{-1}$, it is enough to show that for all bounded continuous $\phi$, 
	 $$
	 \mathbbm{W}^0\circ\left(W^0, \mathbbm{E}^k[\phi(\hat{X}^n)|\mathcal{F}^{n,W^0}]\right)^{-1}
	 \longrightarrow
	 \mathbbm{W}^0\circ\left(W^0, \mathbbm{E}[\phi(\hat{X}^n)|\mathcal{F}^{n,W^0}]\right)^{-1}.
	 $$
As  $\mathbbm{E}^k[\phi(\hat{X}^n)|\mathcal{F}^{n,W^0}]= \underset{I}{\sum}\frac{\mathbbm{E}^k\left[\phi(\hat{X}^n)\mathds{1}_{S^n_I}(W^0)\right]}{\mathbbm{W}^0[S^n_I]}\mathds{1}_{S^n_I}(W^0)$, for all $k$, we are reduced to prove for all $\phi\in\mathcal{C}_b(\mathcal{X})$, $\psi\in\mathcal{C}_b(\mathbbm{R})$, and $\zeta\in\mathcal{C}_b(\mathcal{W}^0)$ that
	 \begin{equation}\label{EapproxNash1}
	 \begin{array}{rcl}
	 &&
	 \mathbbm{E}^{\mathbbm{W}^0}\left[ \psi\left(\underset{I}{\sum}\frac{\mathbbm{E}^k\left[\phi(\hat{X}^n)\mathds{1}_{S^n_I}(W^0)\right]}{\mathbbm{W}^0[S^n_I]}\mathds{1}_{S^n_I}(W^0)\right)\zeta(W^0)\right]\\
	 &&\underset{k}{\longrightarrow}\;\;
	 \mathbbm{E}^{\mathbbm{W}^0}\left[ \psi\left(\underset{I}{\sum}\frac{\mathbbm{E}\left[\phi(\hat{X}^n)\mathds{1}_{S^n_I}(W^0)\right]}{\mathbbm{W}^0[S^n_I]}\mathds{1}_{S^n_I}(W^0)\right)\zeta(W^0)\right].
	 \end{array}
	 \end{equation}
	 Since $\mathbbm{P}$ and the $\mathbbm{P}^k$ all have the same first marginal $\mathbbm{W}^0$, then the convergence of $\mathbbm{P}^k$ to $\mathbbm{P}$ is a stable convergence in the sense that for all bounded continuous $f$ and bounded Borel $g$, we have that  $\mathbbm{E}^k[f(X)g(W^0)]$ tends to $\mathbbm{E}[f(X)g(W^0)]$, see Lemma 2.1 in \cite{LackerCompatibility} for instance. In particular, by continuity of $\phi$ and $\hat{X}^n$, we have that $\underset{I}{\sum}\frac{\mathbbm{E}^k\left[\phi(\hat{X}^n)\mathds{1}_{S^n_I}(W^0)\right]}{\mathbbm{W}^0[S^n_I]}\mathds{1}_{S^n_I}(W^0)$ tends   $\mathbbm{W}^0$ a.s. to $\underset{I}{\sum}\frac{\mathbbm{E}\left[\phi(\hat{X}^n)\mathds{1}_{S^n_I}(W^0)\right]}{\mathbbm{W}^0[S^n_I]}\mathds{1}_{S^n_I}(W^0)$, and by the dominated convergence Theorem, \eqref{EapproxNash1} holds for any $\phi,\psi,\zeta$, implying the desired continuity of the mapping $\Phi_n$.
	 \\
	 \\
	 We now show that $\Phi_n$ takes values in $\Pi^0_c$, see Notation \ref{E_discret_set}. Let $\mathbbm{P}\in\mathcal{P}(\Pi^0)$ and $\mathbbm{Q}:=\Phi_n(\mathbbm{P})=\mathbbm{W}^0\circ\left(W^0, \mathbbm{P}\circ(\hat{X}^n|\mathcal{F}^{n,W^0})^{-1}\right)^{-1}$. It is immediate that under $\mathbbm{Q}$, $W^0$ is an $\mathbbm{F}^{W^0}$-Brownian motion, however, in order to fit the definition of $\Pi^0_c$ which is included in $\Pi^0$, we need to show that $W^0$ is an $\mathbbm{F}^{M,W^0}$-Brownian motion. Since $M$ is  $\mathbbm{Q}$ a.s. equal to the $\mathcal{F}^{W^0}$-measurable random measure $\mathbbm{P}\circ(\hat{X}^n|\mathcal{F}^{n,W^0})^{-1}$, in order to show that $W^0$ is indeed an $\mathbbm{F}^{M,W^0}$-Brownian motion, it is enough to show that $\mathbbm{P}\circ(\hat{X}^n|\mathcal{F}^{n,W^0})^{-1}$ is $\mathbbm{F}^{W^0}$-adapted in the sense that for any $F\in\mathcal{F}^{X}_t$, $\mathbbm{P}\circ(\hat{X}^n|\mathcal{F}^{n,W^0})^{-1}(F)$ is $\mathcal{F}^{W^0}_t$-measurable.

We fix some $k< 2^n$, $t\in[t_k,t_{k+1}[$ and $F\in\mathcal{F}^{X}_t$. By construction of $\hat{X}^n$, we have that
\begin{equation}\label{E_discret_2}
\{\hat{X}^n\in F\}\in \mathcal{F}^{X}_{t_k}.
\end{equation}
	Then,  by definition of $\mathcal{P}(\Pi^0)$, see Definition \ref{def_admissible},  $W^0$ is under  $\mathbbm{P}$ and $\mathbbm{F}$-Brownian motion, so for all $t$, $\mathcal{F}^{X}_t$ is conditionally independent of $\mathcal{F}^{W^0}_T$ given $\mathcal{F}^{W^0}_t$, and in particular, combining \eqref{E_discret_2} and Theorem 3.11 in \cite{LackerCompatibility} we have 
\begin{equation}\label{E_discret_3}
\mathbbm{P}\circ(\hat{X}^n\in F|\mathcal{F}^{W^0}_T)^{-1}= \mathbbm{P}\circ(\hat{X}^n\in F|\mathcal{F}^{W^0}_{t_k})^{-1} \text{ a.s.}
\end{equation}	

Then, we can write
	  \begin{equation}
	  	\begin{array}{rcl}
	  	\mathbbm{P}\circ(\hat{X}^n|\mathcal{F}^{n,W^0})^{-1}[F]&:=&\mathbbm{P}[\hat{X}^n\in F|\mathcal{F}^{n,W^0}_T]\\
	  	&=&\mathbbm{E}[\mathbbm{P}[\hat{X}^n\in F|\mathcal{F}^{W^0}_T]|\mathcal{F}^{n,W^0}_T]\\
	  	&=&\mathbbm{E}[\mathbbm{P}[\hat{X}^n\in F|\mathcal{F}^{W^0}_{t_k}]|\mathcal{F}^{n,W^0}_T]\\
	  	&=&\mathbbm{E}[\mathbbm{P}[\hat{X}^n\in F|\mathcal{F}^{W^0}_{t_k}]|\mathcal{F}^{n,W^0}_{t_k}]\\
	  	&=&\mathbbm{P}[\hat{X}^n\in F|\mathcal{F}^{n,W^0}_{t_k}]\\
	  	&=&\mathbbm{P}[\hat{X}^n\in F|\mathcal{F}^{n,W^0}_{t}]
	  	\end{array}
	  \end{equation}
where the third  equality holds by \eqref{E_discret_3}, and the fourth one by independence of the increments of $W^0$, and construction of $\mathbbm{F}^{n,W^0}$.
So we indeed have that $\mathbbm{P}\circ(\hat{X}^n|\mathcal{F}^{n,W^0})^{-1}(F)$ is $\mathcal{F}^{W^0}_t$-measurable, and therefore, $W^0$ is under $\mathbbm{Q}$ an $\mathbbm{F}^{M,W^0}$-Brownian motion so that $\mathbbm{Q}\in\Pi^0$.

We conclude showing that $\mathbbm{Q}$ verifies \eqref{E_discret_set}. We fix an integer $k$, and we have that
\begin{equation}
	\begin{array}{rcl}
	 \mathbbm{E}^{\mathbbm{Q}}[M[\bar{K}_k]]&=&\mathbbm{E}^{\mathbbm{Q}}[\mathbbm{P}\circ(\hat{X}^n|\mathcal{F}^{n,W^0})^{-1}[\bar{K}_k]]\\
	 &=&\mathbbm{E}^{\mathbbm{P}}[\mathbbm{P}[\hat{X}^n\in \bar{K}_k|\mathcal{F}^{n,W^0}]]\\
	 &=&\mathbbm{P}[\hat{X}^n\in \bar{K}_k]\\
	 &\geq&\mathbbm{P}[\hat{X}^n\in K^n_k]\\
	 &\geq&\mathbbm{P}[X\in K^{\infty}_k]\\
	 &\geq &1- \frac{1}{k},
	\end{array}
\end{equation}
where the last inequality holds since $\mathbbm{P}\in\mathcal{P}(\Pi^0)$, hence $\mathbbm{P}\circ X^{-1}\in \mathcal{K}^{X}$ and by construction of the sets $\bar{K}_k$, $K^n_k$ and $K_k$.

We may now conclude with a version of the Kakutani's Theorem. 
We consider the restriction of $\mathcal{P}^*$ on $\Pi^0_c$\\ $\mathcal{P}^*: \Pi^0_c\xtwoheadrightarrow{  }\mathcal{P}(\Pi^0_c)$ which defines an uhc correspondence taking non empty compact convex values (see Proposition \ref{PropReg} Item 3). 

We recall that  $\Phi_n:\mathcal{P}(\Pi^0_c)\longrightarrow \Pi^0_c$ is a continuous mapping, and that $\Pi^0_c$ is a convex compact subset of a locally convex topological space (see Lemma \ref{L_discret_set}), so by Theorem \ref{Kakutani} and Lemma \ref{Lkakutani}, there exists in  $\Pi^0_c$ a fixed point $\pi^*_n\in\Phi_n\circ\mathcal{P}^*(\pi^*_n)$.

We conclude this proof by showing that if we set $\mathbbm{P}_n^*$ to be the element of $\mathcal{P}^*(\pi^*_n)$ such that $\pi^*_n= \Phi(\mathbbm{P}_n^*)$, then $\mathbbm{P}_n^*$ is a discretized strong Nash equilibrium of order $n$, see Definition \ref{D_discret}.

$\mathbbm{P}_n^*$ belongs to $\mathcal{P}(\Pi^0)$ and  $\mathcal{P}^*(\Pi^0)$. Moreover, it verifies $\mathbbm{P}_n^*\circ(W^0,M)^{-1}=\pi^*_n = \mathbbm{W}^0\circ\left(W^0, \mathbbm{P}_n^*(\hat{X}^n|\mathcal{F}^{n,W^0})^{-1}\right)^{-1}$ hence $M = \mathbbm{P}_n^*(X|\mathcal{F}^{n,W^0})^{-1}\quad$ $\mathbbm{P}_n^*$ a.s. meaning that \eqref{E_discret_0} holds,
%
and $\mathbbm{P}_n^*$ is a discretized strong Nash equilibrium of order $n$.
\end{prooff}

\subsection{Existence of a weak Nash equilibrium}\label{S3.4}

We conclude this section by proving Theorem \ref{T_existence}, i.e. the existence of a weak Nash equilibrium.

\begin{prooff}{} {\bf \hspace{-5mm} of Theorem \ref{T_existence}}\quad
	For every $n\in\mathbbm{N}$, we consider $\mathbbm{P}_n^{*}$ a discretized strong Nash equilibrium of order $n$ whose existence is ensured by Proposition \ref{P_discret}. Every $\mathbbm{P}_n^{*}$ belongs to $\mathcal{P}(\Pi^0_c)$ which is compact since $\Pi^0_c$ is (see Lemmas \ref{L3} (i) and \ref{L_discret_set}). So we may consider an accumulation point $\mathbbm{P}^*\in \mathcal{P}(\Pi^0_c)$ of the sequence $(\mathbbm{P}_n^{*})_n$. We will now show that $\mathbbm{P}^*$ is a weak solution of the MFG in the sense of Definition \ref{weaksol}. 
	
	We first remark that, since every $\mathbbm{P}_n^{*}$ belongs to $\mathcal{P}^*(\Pi^0)$ which is closed (see Proposition \ref{PropReg} (iii), then $\mathbbm{P}^*$ also belongs to $\mathcal{P}^*(\Pi^0)$, which means that $\mathbbm{P}^*$ satisfies the individual optimality condition of Definition \ref{weaksol}. We are left to show that $\mathbbm{P}^*$ satisfies the weak equilibrium condition of Definition \ref{weaksol}. In the sequel, we still denote $(\mathbbm{P}_n^{*})_n$ the subsequence which converges to  $\mathbbm{P}^{*}$. 
	
	We need to show that $M = \mathbbm{P}^{*}\circ(X|\mathcal{F}^{M,W^0})^{-1},$ $\mathbbm{P}^{*}$ a.s. This means that for all $F\in\mathcal{F}^X$, $M(F)= \mathbbm{P}^{*}[X\in F|\mathcal{F}^{M,W^0}]$, $\mathbbm{P}^{*}$ a.s. By approximation it is enough to show that $M(\phi)= \mathbbm{P}^{*}[\phi(X)|\mathcal{F}^{M,W^0}]$, $\mathbbm{P}^{*}$ a.s. for any bounded continuous $\phi$, and by the functional monotone class theorem (see Theorem 19 in \cite{dellmeyer75} Chapter I), it is enough to show that for any $N$, $t_1,\cdots, t_N$, $\phi_1,\cdots,\phi_N\in\mathcal{C}_b(\mathbbm{R}^d)$, $\psi\in\mathcal{C}_b(\mathfrak{M}_+^1(\mathcal{X}))$, and $F\in\mathcal{F}^{W^0}$, we have:
	
\begin{equation}\label{EweakNash0}
	\mathbbm{E}^{\mathbbm{P}^{*}}\left[M\psi(M)\mathds{1}_F(W^0)\underset{i\leq N}{\Pi}\phi_i(X_{t_i})\right]
	= \mathbbm{E}^{\mathbbm{P}^{*}}\left[\psi(M)\mathds{1}_F(W^0)\underset{i\leq N}{\Pi}\phi_i(X_{t_i})\right].
\end{equation}
	For every $n$, we have that $M=\mathbbm{P}_n^{*}[\hat{X}^n|\mathcal{F}^{n,W^0}]$.
	In particular, $M$ is a.s. equal to an $\mathcal{F}^{n,W^0}$-measurable random measure, and $M = \mathbbm{P}_n^{*}(\hat{X}^n|\mathcal{F}^{n,W^0}\vee \mathcal{F}^{M})^{-1}$,  $\mathbbm{P}_n^{*}$ a.s., implying that for all $n\geq 0$ and $F\in  \mathcal{F}^{n,W^0}$, 
		\begin{equation}\label{EweakNash1}
			\mathbbm{E}^{\mathbbm{P}_n^{*}}\left[M\psi(M)\mathds{1}_F(W^0)\underset{i\leq N}{\Pi}\phi_i(X_{t_i})\right]
			= \mathbbm{E}^{\mathbbm{P}_n^{*}}\left[\psi(M)\mathds{1}_F(W^0)\underset{i\leq N}{\Pi}\phi_i(\hat{X}^n_{t_i})\right].
		\end{equation}
Since $\mathcal{F}^{n,W^0}$ is increasing in $n$, then for fixed $F\in\mathcal{F}^{n,W^0}$, \eqref{EweakNash1} above also holds under $\mathbbm{P}_k^{*}$ for all $k\geq n$. By the stable convergence of  $\mathbbm{P}_k^{*}$ to $\mathbbm{P}^*$, the left hand side of \eqref{EweakNash1} tends to the left hand side of \eqref{EweakNash0}. So in order to show that \eqref{EweakNash0} holds for this specific $F\in\mathcal{F}^{n,W^0}$, we will show that 
 \begin{equation}\label{EweakNash2}
	\mathbbm{E}^{\mathbbm{P}_k^{*}}\left[\psi(M)\mathds{1}_F(W^0)\underset{i\leq N}{\Pi}\phi_i(\hat{X}^k_{t_i})\right]
	\underset{k}{\longrightarrow}
	\mathbbm{E}^{\mathbbm{P}^{*}}\left[\psi(M)\mathds{1}_F(W^0)\underset{i\leq N}{\Pi}\phi_i(X_{t_i})\right].
\end{equation}
	We fix $\epsilon >0$. Since $(\mathbbm{P}_k^{*})_k$ is tight, we may fix a compact subset $K_{\epsilon}$ of $\mathcal{X}$ such that $\mathbbm{P}_k^{*}(\mathcal{X}\backslash K_{\epsilon})\leq \epsilon$ for all $k$, and such that $\hat{X}^k$ converges uniformly to $X$ on $K_{\epsilon}$. Eventually, $X$ and all the $\hat{X}^n$ are uniformly bounded by some constant $C>0$ on this $K_{\epsilon}$, and all the  $\phi_i$ are uniformly continuous on the closed ball $\bar{B}(0,C)$. In particular, there exists $k_0$ such that for all $k\geq k_0$, and $\omega\in K_{\epsilon}$, 
	\begin{equation}
		\left|\underset{i\leq N}{\Pi}\phi_i(\hat{X}^k_{t_i}(\omega))-\underset{i\leq N}{\Pi}\phi_i(\omega(t_i))\right|\leq \epsilon.
	\end{equation}
This implies that
	\begin{equation}
		\begin{array}{rcl}
		&&\hspace{-8mm}
		\left|\mathbbm{E}^{\mathbbm{P}_k^{*}}\left[\psi(M)\mathds{1}_F(W^0)\underset{i\leq N}{\Pi}\phi_i(\hat{X}^k_{t_i})\right]
		-\mathbbm{E}^{\mathbbm{P}^{*}}\left[\psi(M)\mathds{1}_F(W^0)\underset{i\leq N}{\Pi}\phi_i(X_{t_i})\right]\right|\\
		&&\hspace{-8mm}\leq
		\left| \mathbbm{E}^{\mathbbm{P}_k^{*}}\left[\psi(M)\mathds{1}_F(W^0)\underset{i\leq N}{\Pi}\phi_i(\hat{X}^k_{t_i})\right]
		-\mathbbm{E}^{\mathbbm{P}_k^{*}}\left[\psi(M)\mathds{1}_F(W^0)\underset{i\leq N}{\Pi}\phi_i(X_{t_i})\right]\right|\\
		&&\hspace{-3mm}+ \left| \mathbbm{E}^{\mathbbm{P}_k^{*}}\left[\psi(M)\mathds{1}_F(W^0)\underset{i\leq N}{\Pi}\phi_i(X_{t_i})\right]
		-\mathbbm{E}^{\mathbbm{P}^{*}}\left[\psi(M)\mathds{1}_F(W^0)\underset{i\leq N}{\Pi}\phi_i(X_{t_i})\right]\right|.
		\end{array}
	\end{equation}
	It is immediate that the second term tends to zero, and for the first one we have for all $k\geq k_0$:
	\begin{equation}
	\begin{array}{rcl}

	&&\left| \mathbbm{E}^{\mathbbm{P}_k^{*}}\left[\psi(M)\mathds{1}_F(W^0)\underset{i\leq N}{\Pi}\phi_i(\hat{X}^k_{t_i})\right]
	-\mathbbm{E}^{\mathbbm{P}_k^{*}}\left[\psi(M)\mathds{1}_F(W^0)\underset{i\leq N}{\Pi}\phi_i(X_{t_i})\right]\right|\\
	&\leq&\|\psi\|_{\infty}\mathbbm{E}^{\mathbbm{P}_k^{*}}\left[\left|\underset{i\leq N}{\Pi}\phi_i(\hat{X}^k_{t_i})-\underset{i\leq N}{\Pi}\phi_i(X_{t_i})\right|\right]\\
	&\leq&\|\psi\|_{\infty}\mathbbm{E}^{\mathbbm{P}_k^{*}}\left[\mathds{1}_{K_{\epsilon}}\left|\underset{i\leq N}{\Pi}\phi_i(\hat{X}^k_{t_i})-\underset{i\leq N}{\Pi}\phi_i(X_{t_i})\right|\right]
	\\
	&&+\|\psi\|_{\infty}\mathbbm{E}^{\mathbbm{P}_k^{*}}\left[\mathds{1}_{\mathcal{X}\backslash K_{\epsilon}}\left|\underset{i\leq N}{\Pi}\phi_i(\hat{X}^k_{t_i})-\underset{i\leq N}{\Pi}\phi_i(X_{t_i})\right|\right]\\
	&\leq& 2^N\epsilon\,\|\psi\|_{\infty}\underset{i\leq N}{\Pi}\|\phi_i\|_{\infty}+\epsilon\,\|\psi\|_{\infty}.
	\end{array}
	\end{equation}
	Since we may pick $\epsilon$ as small as we want, then we indeed have that \eqref{EweakNash2} holds and therefore that \eqref{EweakNash1} holds for any $F\in\mathcal{F}^{n,W^0}$. Since this is true for any $n$, then \eqref{EweakNash0} holds for any $F\in\bigcup_n\mathcal{F}^{n,W^0}$.
		
	$\bigcup_n\mathcal{F}^{n,W^0}$ is stable by finite intersection hence forms a $\pi$-system, see Definition 4.9 in \cite{aliprantis}. The sets of $F\in\mathcal{F}^{W^0}$ verifying \eqref{EweakNash0} form a monotone class (also called $\lambda$-system, see Definition 4.9 in \cite{aliprantis} again), so by the monotone class Theorem (or Dynkin's Lemma, see 4.11 in \cite{aliprantis}), we have that \eqref{EweakNash0} holds for all $F\in\sigma\left(\bigcup_n\mathcal{F}^{n,W^0}\right)$ which is equal to $\mathcal{F}^{W^0}$, see Remark \ref{R_discret} Item 1, and the proof is complete.
\end{prooff}


\section{McKean-Vlasov second order backward SDEs}\label{SMkV2BSDE}

From now on, we specialize the discussion to the no common noise context, i.e. $p_0=0$ and $W=W^1$. Consequently the distribution of $X$ is now deterministic as it is not conditioned anymore on the common noise. We shall work on the smaller canonical space $\Omega=\mathcal{X}\times \mathcal{Q}$ by appropriate projection of $\mathcal{W}$.


 In particular, notice that in the present context, the notions of weak and strong solutions of the MFG coincide.

This section contains the second main results of the paper. Our objective is to provide a characterization of the solution of the MFG in the no common noise context by means of a McKean-Vlasov second order backward SDE (2BSDE). This requires a non-degeneracy condition obtained by separating the control of the drift and the one of the diffusion coefficient. We therefore introduce two control sets $A$ and $B$ where the drift control process and the diffusion control process take values, respectively.

We denote by $\mathcal{Q}^A$ the set of relaxed controls, i.e. of measures $q$ on $[0,T]\times A$ such that $q(\cdot\times A)$ is equal to the Lebesgue measure. Each $q\in\mathcal{Q}^A$ may be identified with a measurable function $t\mapsto q_t$ from $[0,T]$ to $\mathfrak{M}_+^1(A)$ determined a.e. by $q(dt,da)=q_t(da)dt$.

We define similarly the set of relaxed controls $\mathcal{Q}^B$ by replacing the space $A$ with $B$, and we denote $\mathcal{Q}:= \mathcal{Q}^A\times \mathcal{Q}^B$ with corresponding canonical process $Q:=(Q^A,Q^B)$.
	
As in the previous section, we equip these spaces with their natural filtrations. We also introduce the right-continuous filtration $\mathbbm{F}^{X,+}$ defined for all $t\in[0,T]$  by $\mathcal{F}_t^{X,+}:=\bigcap_{n\geq 0}\mathcal{F}^X_{t+\frac{1}{n}}$.

We denote by $\mathfrak{SM}$ the set of 	all $\mathbbm{P}\in \mathfrak{M}_+^1(\mathcal{X})$ such that $X$ is a $\mathbbm{P}$-semimartingale  with absolutely continuous bracket. By Karandikar \cite{karandikar}, there exists an $\mathbbm{F}^X$-progressively measurable process, denoted by $\langle X\rangle$, which coincides with the quadratic variation of $X$, $\mathbbm{P}$-a.s. for every  $\mathbbm{P}\in \mathfrak{SM}$. We may then introduce the  process $\hat{\sigma}^2$ defined by 
	$$
	\hat{\sigma}^2_t 
	:= \limsup_{\epsilon\searrow 0}\frac{\langle X\rangle_t - \langle X\rangle_{t-\epsilon}}{\epsilon},\quad t\in[0,T].
	$$
	This process is progressively measurable and takes values in the set of  $d \times d$ non-negative symmetric matrices denoted $\mathbbm{S}_d^+$.

We now fix $\mathcal{P}\subset\mathfrak{SM}$. For all $\mathbbm{P}\in \mathcal{P}$, and $t\in[0,T]$ we denote by $\mathcal{F}_t^{X,+,\mathbbm{P}}$ the $\sigma$-field $\mathcal{F}_t^{X,+}$ augmented with $\mathbbm{P}$-null sets, and we denote by $\mathbbm{F}^{X,+,\mathcal{P}}$ the filtration given by 
$$\mathcal{F}^{X,+,\mathcal{P}}_t:=\underset{\mathbbm{P}\in\mathcal{P}}{\bigcap}\mathcal{F}^{X,+,\mathbbm{P}}_t,\quad t\in[0,T].$$

We say that a property holds $\mathcal{P}-$quasi surely (abbreviated as $\mathcal{P}-$q.s.) if it holds $\mathbbm{P}-$a.s. for all $\mathbbm{P}\in\mathcal{P}$. We also denote by $\mathbbm{S}^2(\mathcal{P})$ the collection of all c\`adl\`ag $\mathbbm{F}^{X,+,\mathcal{P}}$-adapted processes $S$ with 
 $$
 \big\| S \big\|_{\mathbbm{S}^2(\mathcal{P})}^2
 \;:=\;
 \sup_{\mathbbm{P}\in\mathcal{P}} \mathbbm{E}^{\mathbbm{P}}\Big[\sup_{t\le T}\,S_t^2\Big]<\infty.
 $$

Finally,  we denote by $\mathbbm{H}^2(\mathcal{P})$ the collection of all $\mathbbm{F}^{X,+,\mathcal{P}}-$progressively measurable processes $H$ with
 $$
 \big\| H \big\|_{\mathbbm{H}^2(\mathcal{P})}^2
 \;:=\;
 \sup_{\mathbbm{P}\in\mathcal{P}}
 \mathbbm{E}^{\mathbbm{P}}\Big[ \int_0^T H_t^\intercal d\langle X\rangle_tH_t \Big]
 \;=\;
 \sup_{\mathbbm{P}\in\mathcal{P}}
 \mathbbm{E}^{\mathbbm{P}}\Big[ \int_0^T H_t^\intercal \hat\sigma_t^2H_t dt\Big]
 \;<\;
 \infty.
 $$

\subsection{Controlled state process}

For a fixed $m\in\mathfrak{M}_+^1(\mathcal{X})$, the controlled state is defined by the relaxed SDE
\begin{equation}\label{controlled2}
 X_t
 = 
 X_0 + \int_{\!0}^t\!\!\int_{\!\!A\times B}\!\!(\sigma_r\lambda_r)(X,m,a,b)Q_r(da,db)dr 
                                + \!\!\int_{\!B}\!\!\sigma_r(X,m,b)N^{B}(db,dr),
\end{equation}
where $N^B$ is a martingale measure with intensity $Q^B_tdt$, 
 $$
 \lambda:[0,T]\times \mathcal{X}\times \mathfrak{M}_+^1(\mathcal{X})\times A
 \longrightarrow \mathbbm{R}^d,
 ~~
 \sigma: [0,T]\times \mathcal{X}\times \mathfrak{M}_+^1(\mathcal{X})\times B
 \longrightarrow \mathbbm{M}_{p,d}(\mathbbm{R}),
 $$
are progressively measurable maps (in the sense detailed in Subsection \ref{sec:controlled}). The generator of our controlled martingale problem is defined for $\phi\in\mathcal{C}^2_b(\mathbbm{R}^d)$, $(a,b)\in A\times B$, and $(t,x,y)\in [0,T]\times\mathcal{X}\times\mathbbm{R}^d$ by
 $$
 \mathcal{A}^{a,b,m}_{t,x}\phi(y)
 :=
 (\sigma_t\lambda_t)(x,m,a,b)\cdot D \phi(y)+ \frac{1}{2}
 \sigma_t
 \sigma^{\intercal}_t(x,m,b):D^2\phi(y).
$$

\begin{definition}\label{def_admissible4} Fix some $q_0\in A\times B$, and denote $Q^0$ the measure defined by $Q^0_t=\delta_{q_0}$, $t\in[0,T]$. For $(s, x)\in[0,T]\times\mathcal{X}$ and $m\in\mathfrak{M}_+^1(\mathcal{X})$, we denote
\\
{\rm (i)} $\overline{\mathcal{P}}^{m}_{s, x}$ the subset of all $\mathbbm{P}\in \mathfrak{M}_+^1(\Omega)$ s.t. $\mathbbm{P}[(X_{\wedge s},Q_{\wedge s}) =  (x_{\wedge s},Q^0_{\wedge s})]=1$, and 
	$$\phi(X_{t})-\int_s^{t}\!\!\!\int_{A\times B} \mathcal{A}^{a,b,m}_{r,X}\phi(X_r)\, Q_r(da,db)dr,\quad t\in[s,T],
	$$
is a $(\mathbbm{P},\mathbbm{F})$-martingale for all $\phi\in\mathcal{C}^2_b(\mathbbm{R}^d)$;
\\
{\rm (ii)} $\overline{\mathcal{M}}^{m}_{s, x}$ the subset of all $\mathbbm{P}\in \mathfrak{M}_+^1(\Omega)$ s.t. $\mathbbm{P}[(X_{\wedge s},Q_{\wedge s}) =  (x_{\wedge s},Q^0_{\wedge s}]=1$, and, 
	$$
	\phi(X_{t})-\frac{1}{2}\int_s^{t}\int_{B} 
		\sigma_t\sigma^{\intercal}_t(x,m,b):D^2\phi(X_r)\, Q^B_r(db)dr,\quad t\in[s,T],
	$$
is a $(\mathbbm{P},\mathbbm{F})$-martingale for all $\phi\in\mathcal{C}^2_b(\mathbbm{R}^d).$
\\
\\
For any $s,x,m$, we set $\mathcal{P}^m_{s,x}:=\{\mathbbm{P}\circ X^{-1}:\, \mathbbm{P}\in \overline{\mathcal{P}}^{m}_{s,x}\}$ and $\mathcal{M}^m_{s,x}:=\{\mathbbm{P}\circ X^{-1}:\, \mathbbm{P}\in \overline{\mathcal{M}}^{m}_{s,x}\}$.
\\
Finally, we simply denote $\overline{\mathcal{M}}^{m}:=\overline{\mathcal{M}}^{m}_{0,0}$, $\overline{\mathcal{P}}^{m}:=\overline{\mathcal{P}}^{m}_{0,0}$,  $\mathcal{M}^{m}:=\mathcal{M}^{m}_{0,0}$ and $\mathcal{P}^{m}:=\mathcal{P}^{m}_{0,0}$.
\end{definition}

\subsection{Solving a McKean-Vlasov 2BSDE}

Similar to the previous sections, let $\xi:\mathcal{X}\rightarrow \mathbbm{R}$ be a random variable, and $f:[0,T]\times \mathcal{X}\times \mathfrak{M}_+^1(\mathcal{X})\times A\times B\longrightarrow \mathbbm{R}$ a progressively measurable process, and denote the dynamic version of the value function of the individual optimization problem for all $(t,x,m)\in [0,T]\times \mathcal{X}\times \mathfrak{M}_+^1(\mathcal{X})$ by:
 $$
 V^{m}_t(x)
 \;:=\;
  \sup_{\mathbbm{P}\in\overline{\mathcal{P}}^{m}_{t, x}}
  \mathbbm{E}^{\mathbbm{P}}
  \left[\xi
         +\int_t^T\!\!\!\int_{\!\!A\times B}\!\!\!f_r(m,a,b)Q_r(da,db)dr\right].
 $$
The backward SDE characterization of the solution of the MFG requires to introduce the following nonlinearity: 
\begin{equation}\label{def_F}
		F_t(x,z,\Sigma,m)
	:=
	\!\!\!\!
	\sup_{{
			q\in\mathbf{Q}_t(x,\Sigma,m)
	}}
	H_t(x,z,m,q),
	~\quad
	H_{\cdot}(\cdot,z,\cdot,q)
	:=\!\!
	\int_{A\times B}\!\!(f+z\!\cdot\!\sigma\lambda) dq.
\end{equation}
For all $(t,x,z,\Sigma,m)\in[0,T]\times\mathcal{X}\times \mathbbm{R}^d\times\mathbbm{S}_d^+\times \mathfrak{M}_+^1(\mathcal{X})$, where
 \begin{equation}\label{Qc}
 \mathbf{Q}_t(x,\Sigma,m)
 :=
 \Big\{ q\in \mathfrak{M}_1^+(A)\otimes\mathfrak{M}_1^+(B): \int_B\sigma_t\sigma^{\intercal}_t(x,m,b)q^B(db)= \Sigma \Big\}.
 \end{equation}
The following condition is a restatement of Assumption \ref{H_existence} in the present context, with a sufficient condition for the wellposedness of the controlled SDE.

\begin{assumption}\label{A5}\
	\begin{itemize}
		\item $\xi,f,\lambda,\sigma$ are bounded;
		\item $\xi$ and $f_t,\lambda_t,\sigma_t$ for for all $t$, are continuous;
		\item $\lambda,\sigma$ are locally Lipschitz continuous in $x$ uniformly in $(t,a)$ at fixed $m$.
	\end{itemize}
\end{assumption}

\noindent We are now ready for our main characterization of a solution of the MFG from Theorem \ref{T_existence} in terms of the McKean-Vlasov second order backward SDE.

\begin{definition}\label{MkV2BSDE}
	We  say that $(m,Y,Z)\in \mathfrak{M}_+^1(\mathcal{X})\times\mathbbm{S}^2\big(\mathcal{P}^{m}\big)\times\mathbbm{H}^2\big(\mathcal{P}^{m}\big)$ 
	solves the \textbf{McKean-Vlasov 2BSDE}
	\begin{equation}\label{EqMkV2BSDE}
	Y_t = \xi +\int_t^T F_r(X,Z_r,\hat{\sigma}_r^2,m)dr - \int_t^TZ_rdX_r +U_T-U_t,~ t\in[0,T],~\mathcal{P}^m-\mbox{q.s.}
	\end{equation}
	if the following holds.
	\begin{enumerate}
		\item the process $U := Y_{\cdot} - Y_0 +\int_0^{\cdot}F_r(Z_r,\hat{\sigma}_r^2,m)dr - \int_0^{\cdot}Z_rdX_r $ is 
		is a $\mathbbm{P}$-càdlàg  supermartingale, orthogonal to $X$ for every $\mathbbm{P}\in\mathcal{P}^{m}$; 
		\item $m\in \mathcal{P}^{m}$ and $U$ is  an $m$-martingale.
	\end{enumerate}
\end{definition}
Notice that \eqref{EqMkV2BSDE} differs from the the notion introduced in \cite{soner_touzi_zhang2} and further developed in \cite{ptz,lrty} by the fact that both the nonlinearity and the set of probability measures depend on the law of $X$, denoted $m$. We emphasize that $m$ should not be understood as the law of $X$ under arbitrary $\mathbbm{P}\in \mathcal{P}^{m}$. Instead, $m$ denotes the "optimal" measure in $\mathcal{P}^m$, i.e. the one under which $U$ is a martingale.
In other words: the law $m$ which parametrizes the 2BSDE coincides with the optimal law for $X$ within the set of measures under which the 2BSDE holds.

We now state the main result of this second part of the paper, which proof is postponed to Section \ref{S6}.
\begin{theorem}\label{ThMkV}
	Let Assumption \ref{A5} hold true. Then, there exists a solution $(m,Y,Z)$ to the McKean-Vlasov 2BSDE \eqref{EqMkV2BSDE}.
	\\
	Moreover, $m$ is a solution of the Mean-Field game  with coefficients $\sigma\lambda,\sigma,f,\xi$ and $Y=V^m$ meaning that $
	Y_t(x) = V^{m}_t(x)$, for all $(t,x)\in[0,T]\times\mathcal{X}$.
\end{theorem}

\section{2BSDE representation of relaxed controlled problems}\label{S2BSDE}

The aim of this section is to introduce the tools needed for the proof of Theorem \ref{ThMkV}.
We keep working with the spaces introduced at the beginning of the previous section.
However, since marginal distribution $m$ is fixed throughout, we shall drop the dependence on this parameter throughout this section.

\subsection{Controlled state process, optimization problem and value function}

The controlled state process is defined by the relaxed SDE \eqref{controlled2}, and the dynamic version of the value function of this control problem is defined by setting for any $(s,  x) \in [0,T]\times \mathcal{X}$:
\begin{equation}\label{prob:2BSDE}
V_s( x)
:= 
\sup_{\mathbbm{P}\in\overline{\mathcal{P}}_{s, x}}
 \!\!J_s(\mathbbm{P}),
~\mbox{where}~
J_s(\mathbbm{P}):=\mathbbm{E}^{\mathbbm{P}}\!\!\left[\xi+\!\!\int_s^T\!\!\!\!\int_{A\times B}f_r(X,a,b)Q_r(da,db)dr\right],
\end{equation} 
where $\xi,f$ are jointly measurable, with $f$ progressively measurable in $(t,x)$, the spaces of probability measures $\overline{\mathcal{P}}_{s, x}, \overline{\mathcal{P}},\mathcal{M}_{s, x}, \mathcal{M}, \mathcal{P}_{s, x}, \mathcal{P}$ are defined as in Definition \ref{def_admissible4}, with dependence on $m$ dropped throughout.

\begin{proposition}\label{Vcontinuous}
	Under Assumption \ref{A5}, the set-valued map 
	$(s, x)\longmapsto \overline{\mathcal{P}}_{s, x}$ is a compact valued continuous correspondence, $V$ is continuous on $[0,T]\times\mathcal{X}$, and existence holds for the problem \eqref{prob:2BSDE}.
\end{proposition}

\begin{proof}
The compactness of  $\overline{\mathcal{P}}_{s, x}$ is a consequence of Proposition \ref{PropReg} (ii).
Notice that the correspondence $\Gamma:(s,x)\in[0,T]\times \mathcal{X}\longmapsto\{(s,x)\}\times \mathfrak{M}_+^1(\mathcal{Q})$  is continuous as the product of the continuous mapping $(s,x)\mapsto (s,x)$ and of the constant compact valued (hence continuous) correspondence $(s,x)\mapsto\mathfrak{M}_+^1(\mathcal{Q})$, see Theorem 17.28 in \cite{aliprantis}.
	
	Since $\lambda,\sigma$ are locally Lipschitz in $x$ uniformly in $(t,a,b)$, then for any $\mathbbm{Q}\in \mathfrak{M}_+^1(\mathcal{Q})$ there exists a unique weak solution of the corresponding SDE i.e. a unique $\mathbbm{P}\in \overline{\mathcal{P}}_{s,x}$ such that $\mathbbm{P}\circ Q^{-1}=\mathbbm{Q}$. 
	
	We denote $\phi(s,x,\mathbbm{Q})$ this unique $\mathbbm{P}$. It is clear that $(s,x)\mapsto \overline{\mathcal{P}}_{s,x}$ is equal to $\phi\circ \Gamma$, so by continuity of the composition of continuous correspondences (see Proposition \ref{alipran} Item 4), we are left to show that $\phi$ is continuous.
	
	We fix a converging sequence $(s_n,x_n,\mathbbm{Q}_n)\longrightarrow (s,x,\mathbbm{Q})$ in $[0,T]\times\mathcal{X}\times \mathfrak{M}_+^1(\mathcal{Q})$.
	Since $(x_n)_n$ converges, it is included in a compact subset $C$ of $\mathcal{X}$.
	 For all $n$, $\phi(s_n,x_n,\mathbbm{Q}_n)\circ X^{-1}$ is the law of a process which coincides with $x_n\in C$ on $[0,s_n]$ and which is a semi-martingale with bounded (uniformly in $n$) characteristics on $[s_n,T]$. Hence, adapting the proof of Proposition 6.2 in \cite{paperMPv2}, we have that $(\phi(s_n,x_n,\mathbbm{Q}_n)\circ X^{-1})_n$ is tight. Since $A,B$ are compact sets, then $(\phi(s_n,x_n,\mathbbm{Q}_n))_n$ is also tight. We now show that its only possible limiting point is $\phi(s,x,\mathbbm{Q})$, and the proof of the first statement will be complete. Assume (omitting to extract a converging subsequence) that $\phi(s_n,x_n,\mathbbm{Q}_n)$ tends to some $\mathbbm{P}\in \mathfrak{M}_+^1(\Omega)$. Clearly $\mathbbm{P}\circ Q^{-1}= \mathbbm{Q}$. Since $\phi(s,x,\mathbbm{Q})$ is the unique $\mathbbm{P}\in \overline{\mathcal{P}}_{s,x}$ such that $\mathbbm{P}\circ Q^{-1}=\mathbbm{Q}$, in order to show that $\mathbbm{P}= \phi(s,x,\mathbbm{Q})$ and to conclude, it is enough to show that $\mathbbm{P}\in\overline{\mathcal{P}}_{s,x}$. This is shown exactly as Proposition 6.3 in \cite{paperMPv2}. This shows the continuity of $(s, x)\longmapsto \overline{\mathcal{P}}_{s, x}$. 
	 \\
	 It remains to show that $V$ is continuous. We remark that for all $(s,x)$, we have $V_s(x)= \sup_{\mathbbm{P}\in\overline{\mathcal{P}}_{s, x}}J_0(\mathbbm{P}) - \int_0^sf_r(x,q_0)dr$.
	 Since $\xi,f$ are bounded and $\xi$ and $f_t$ for all $t$ are continuous, then $J_0$ is continuous.  As $(s, x)\longmapsto \overline{\mathcal{P}}_{s, x}$ is continuous and compact valued,  the supremum above is in fact a maximum, and  the Berge maximum theorem (see Theorem \ref{Berge}) states that $(s,x)\mapsto \max_{\mathbbm{P}\in\overline{\mathcal{P}}_{s, x}}J_0(\mathbbm{P})$ is continuous.
	 Finally, the dominated convergence theorem permits to show that $(s,x)\mapsto \int_0^sf_r(x,q_0)$ is continuous, hence $V$ is continuous.
\end{proof}

\subsection{2BSDE solved by the value function}

Recall the notations $F, H$, and $\mathbf{Q}$ introduced in \eqref{def_F}-\eqref{Qc}, again dropping the parameter $m$.

\begin{lemma}\label{qhat}
{\rm (i)} $F$ is jointly measurable, and uniformly Lipschitz in $z$;
\\
{\rm (ii)} There exists a   measurable mapping $\hat{q}:[0,T]\times\mathcal{X}\times \mathbbm{R}^d\times\mathbbm{S}_d^+\longrightarrow \mathfrak{M}_1^+(A)\otimes \mathfrak{M}_1^+(B)$ such that for all $(t,x,z,\Sigma)\in[0,T]\times\mathcal{X}\times \mathbbm{R}^d\times\mathbbm{S}_d^+$: 
$$ 
\hat{q}_t(x,z,\Sigma)\in\mathbf{Q}_t(x,\Sigma)
~\mbox{and}~
F_t(x,z,\Sigma)= H_t\big(x,z,\hat{q}_t(x,z,\Sigma)\big).
$$
\end{lemma}

\begin{proof}
(i) The joint measurability of $f$ follows from (ii), proved below, together with the measurability of $f,\lambda,\sigma$ (hence of $H$), and that of $\hat{q}$. We next observe that $H_t(x,\cdot,q)$ is an affine mapping with slope $\int_{A\times B}\sigma_r(x,b)\lambda_r(x,a)q(da,db)$. 
	In particular, $F_t(x,\cdot,\Sigma)$ is convex as the supremum of affine mappings. Denoting $\partial F_t(x,\cdot,\Sigma)$ its subgradient, since $\mathbf{Q}_t(x,\Sigma)$ is compact and since $q\mapsto H_t(x,z,q)$ is continuous for all $z$, we have (see \cite{hiriart2012fundamentals} Section D. Theorem 4.4.2) for all $z$ that $\partial F_t(x,\cdot,\Sigma)(z)\subset co\left(\left\{\int_{A\times B}\sigma_r(x,b)\lambda_r(x,a)q(da,db):\,q\in \mathbf{Q}_t(x,\Sigma)\right\}\right)$, where $co$ denotes the convex hull. In particular, $\partial F_t(x,\cdot,\Sigma)(z)$ is included in the centered closed ball of radius $\|\sigma\lambda\|_{\infty}$. This implies that the semidirectional derivatives of $ F_t(x,\cdot,\Sigma)$ exist at all $z$ and are bounded by $\|\sigma\lambda\|_{\infty}$, and therefore that this mapping is $\|\sigma\lambda\|_{\infty}$-Lipschitz. 
	\\
(ii) Our aim is to show the existence of a measurable selector for the correspondence $(t,x,z,\Sigma)\longmapsto\text{Arg}\max_{q\in \mathbf{Q}_t(x,\Sigma)} H_t(x,z,q)$.
	Theorems 18.19 and 18.10 in \cite{aliprantis} state that if $H$ is continuous in $q$ for fixed $(t,x,z)$ and measurable in $(t,x,z)$ for fixed $q$, and if $\mathbf{Q}$ is a measurable correspondence with compact values, then such a measurable selector indeed exists.
	
	By boundedness and continuity of $f_t,\lambda_t,\sigma_t$ for all $t$, it is immediate that $H$ verifies the conditions mentioned above.
	It is also clear that  $\mathbf{Q}_t(x,\Sigma)$ is a compact subset of $\mathfrak{M}_1^+(A)\otimes \mathfrak{M}_1^+(B)$ for all $t,x,\Sigma$. So we are left to show that $\mathbf{Q}$ is a measurable correspondence. 
	
	Finally, since $\mathbf{Q}_t(x,\Sigma)=\{q\in\mathfrak{M}_1^+(A)\otimes \mathfrak{M}_1^+(B): h(t,x,\Sigma,q)= 0\}$ with $\mathfrak{M}_1^+(A)\otimes \mathfrak{M}_1^+(B)$ compact and $h:(t,x,\Sigma,q) \mapsto \int_B\sigma\sigma^{\intercal}_t(x,b)q^B(db)-\Sigma$, which is measurable in $(t,x,\Sigma)$ at fixed $q$ and continuous in $q$ at fixed $(t,x,\Sigma)$, then by Corollary 18.8 in \cite{aliprantis}, $\mathbf{Q}$ is indeed measurable, and the proof is complete.
\end{proof}

We next recall the definition of a solution for the 2BSDE:
\begin{equation}\label{2BSDE}
	Y_t = \xi +\int_t^T F_r(Z_r,\hat{\sigma}_r^2)dr - \int_t^TZ_rdX_r +U_T-U_t, \quad \mathcal{P}\text{-q.s.}
\end{equation}
(see for instance \cite{lrty} Definition 3.9 in which the terminal time may be random). We introduce the additional notation
\begin{equation}
	\mathcal{P}_{t,\mathbbm{P}}:=\{\mathbbm{P}'\in\mathcal{P}:\, \mathbbm{P}'\text{ coincides with }\mathbbm{P}\text{ on } \mathcal{F}_t^{X,+}\}.
\end{equation}

\begin{definition}\label{Def2BSDE}
	A pair of processes $(Y,Z)\in\mathbbm{S}^2(\mathcal{P})\times\mathbbm{H}^2(\mathcal{P})$ is a solution of the 2BSDE \eqref{2BSDE}  if the process 
		$$ 
		U_t 
		:= 
		Y_t - Y_0 +\int_0^t F_r(Z_r,\hat{\sigma}_r^2)dr - \int_0^t Z_rdX_r,
		~~t\in[0,T], 
		$$
		is   a $\mathbbm{P}$-c\`adl\`ag  supermartingale, orthogonal to $X$ for all $\mathbbm{P}\in\mathcal{P}$ and if it satisfies the  minimality condition 
		$$U_t = \underset{\mathbbm{P}'\in\mathcal{P}_{t,\mathbbm{P}}}{\mbox{\rm essinf}}^{\mathbbm{P}} \mathbbm{E}^{\mathbbm{P}'}[U_T|\mathcal{F}_t^{X,+,\mathbbm{P}}],\quad t\in[0,T],\quad \mathbbm{P}\text{-a.s.}$$.
\end{definition}

\begin{remark}
	We recall that under the continuum hypothesis, the stochastic integral $\int_t^TZ_rdX_r$ may be defined for all $\omega$ independently of the choice of the probability in $\mathcal{P}$, see Nutz \cite{nutz2012pathwise}. 
\end{remark}

The aim of this subsection is to show the following representation result for the value function.
\begin{theorem}\label{Th2BSDE}
	Under Assumption \ref{A5}, $V\in\mathbbm{S}^2(\mathcal{P})$ and there exists  $Z\in\mathbbm{H}^2(\mathcal{P})$ such that $(V,Z)$ solves the 2BSDE \eqref{2BSDE}.
\end{theorem}

To prove this result, we follow the same argument as in \cite{SonerTouziZhang} introducing
 \begin{equation}\label{def_Y_hat}
 \hat{\mathcal{Y}}_t(x)
 := 
 \sup_{\mathbbm{P}\in\mathcal{M}_{t,x}}
 \mathbbm{E}^{\mathbbm{\mathbbm{P}}}[Y^{t,x,\mathbbm{P}}_t]
 ~~\mbox{for all}~~
 (t,x)\in[0,T]\times\mathcal{X},
\end{equation}
where $(Y^{t,x,\mathbbm{P}},Z^{t,x,\mathbbm{P}})$ is the unique solution of the (well posed) BSDE on the space $(\mathcal{X},\mathcal{F}^X,\mathbbm{F}^{X,+},\mathbbm{P})$:
	\begin{equation}\label{BSDE-P}
		Y^{t,x,\mathbbm{P}}_s 
		= 
		\xi + \int_s^T F_r(Z^{t,x,\mathbbm{P}}_r,\hat{\sigma}_r^2)dr 
		                      - Z^{t,x,\mathbbm{P}}_rdX_r 
		                      - dM^{t,x,\mathbbm{P}}_r,
		~~ s\in[t,T],
	\end{equation}
for some martingale $M^{t,x,\mathbbm{P}}$, with $\langle X,M^{t,x,\mathbbm{P}}\rangle=0$, $\mathbbm{P}-$a.s.

\begin{proposition}\label{PGirs2}
	$V=\hat{\mathcal{Y}}$.
\end{proposition}
\begin{proof}
Denote $f^Q_r:=\int_{A\times B}f_r(a,b)Q_r(da,db)$, $b^Q_r:=\int_{A\times B}\sigma_r(b)\lambda_r(a)Q_r(da,db)$, and fix $(t,x)\in[0,T]\times \mathcal{X}$.
	\\
{\bf 1.} We first prove that $V_t(x)\leq\hat{\mathcal{Y}}_t(x)$. For an arbitrary $\mathbbm{P}\in\overline{\mathcal{P}}_{t,x}$, it follows from Theorem 2.7 in \cite{nicole1987compactification} that there exists an $\mathbbm{F}^X$-progressively measurable process $\bar{q}$ such that the \textit{feedback control} $\mathbbm{P}\circ(X,\bar{q}(X))^{-1}$ belongs to $\overline{\mathcal{P}}_{t,x}$ and  $\mathbbm{E}^{\mathbbm{P}}\left[\xi+\int_t^Tf^Q_rdr\right]=\mathbbm{E}^{\mathbbm{P}}\left[\xi+\int_t^Tf^{\bar{q}(X)}_rdr\right]$.
 
We now work on the filtered space $(\mathcal{X},\mathcal{F}^X,\mathbbm{F}^{X,+})$. Even though $\mathbbm{P}$ is defined on the larger space $(\Omega,\mathcal{F})$, we will often write $\mathbbm{P}$ instead of $\mathbbm{P}\circ X^{-1}$ when there can be no confusion.
By Theorem IV-2 in \cite{EK_Mele}, there exists on a bigger space a martingale measure $N^B$ with intensity $\bar{q}^B(X)_tdt$ such that $$dX_s = b^{\bar{q}(X)}_sds +\int_B \sigma_s(X,b)N^B(db,ds).
$$
Notice that the process 
	$$
	L_s
	\;:=\;
	-\int_t^s \left(\int_A\lambda_r(X,a)\bar{q}^A_r(X)(da)\right) \int_BN^B(db,dr),
	~~s\in[t,T],
	$$
is a continuous martingale with bounded quadratic variation. Then we may introduce the probability measure $G(\mathbbm{P})$ by:
	$$
	\frac{dG(\mathbbm{P})}{d\mathbbm{P}} 
	\;=\; 
	\mathcal{E}(L)
	\;:=\; 
	e^{L-\frac12\langle L\rangle}.
	$$ 
Since $\langle X,L\rangle = -\langle \int_t^{\cdot}\int_B\sigma_r(b)N^B(db,dr),L\rangle=\int_t^{\cdot}b^{\bar{q}(X)}_rdr$, it follows from the Girsanov Theorem that $X$ is a $G(\mathbbm{P})$-martingale with unchanged quadratic variation $\langle X\rangle=\int_t^{\cdot}\int_{ B}\sigma\sigma^{\intercal}_r(X,b)\bar{q}(X)_r^B(db)dr$, $G(\mathbbm{P})$-a.s. Hence $G(\mathbbm{P})\in\mathcal{M}_{t,x}$.
 
Considering on $(\mathcal{X},\mathcal{F}^X,\mathbbm{F}^{X,+},G(\mathbbm{P}))$ the BSDE
	\begin{equation}
\bar{Y}^{t,x,G(\mathbbm{P})}_s = \xi + \int_s^T\left(f^{\bar{q}(X)}_r+\bar{Z}^{t,x,G(\mathbbm{P})}_rb^{\bar{q}(X)}_r\right)dr - \bar{Z}^{t,x,G(\mathbbm{P})}_rdX_r -d\bar{M}^{t,x,G(\mathbbm{P})}_r,
\end{equation}
for $s\in[t,T],$ we will now show that  we have
		\begin{equation}\label{EqGirs}
			\mathbbm{E}^{\mathbbm{P}}\left[\xi+\int_t^Tf^Q_rdr\right]
			\;=\; 
			\mathbbm{E}^{G(\mathbbm{P})}[\bar{Y}^{t,x,G(\mathbbm{P})}_t]
			\;\le\; 
			\mathbbm{E}^{G(\mathbbm{P})}[Y^{t,x,G(\mathbbm{P})}_t],
		\end{equation} 
and this  implies that $V_t(x)\leq\hat{\mathcal{Y}}_t(x)$.
In order to show that the equality in \eqref{EqGirs} holds, we consider under $\mathbbm{P}$ the solution $(\tilde{Y},\tilde{Z},\tilde{M})$ of the  BSDE 
	$$
	\tilde{Y}_s 
	= 
	\xi +\int_s^Tf^{\bar{q}(X)}_rdr - \tilde{Z}_rdX_r+\tilde{Z}_rb^{\bar{q}(X)}_rdr - d\tilde{M}_r,
	\quad s\in [t,T].
	$$ 
As $X-\int_t^{\cdot}b^{\bar{q}(X)}_rdr$ is a $\mathbbm{P}$-martingale, then 
	$\tilde{Y}+\int_t^{\cdot}f^{\bar{q}(X)}_rdr$ is also a $\mathbbm{P}$-martingale, hence by Girsanov Theorem, $\tilde{Y} +\int_t^{\cdot}f^{\bar{q}(X)}_rdr - \langle \tilde{Y}, L\rangle$ is  a $G(\mathbbm{P})$-martingale. 	
	Since $X$ is a $G(\mathbbm{P})$-martingale, we obtain by standard decomposition that
	$$\tilde{Y}_s = \xi +\int_s^Tf^{\bar{q}(X)}_rd_r -\int_s^Td\langle \tilde{Y}, L\rangle_r -\int_s^TZ'_rdX_r + (M'_T-M'_s),\quad s\in[t,T]$$ 
	for some process $Z'$ and some martingale $M'$ orthogonal to $X$.
	 Then $\langle \tilde{Y}, L\rangle =  \int_t^{\cdot}Z'_rd\langle X,L\rangle =- \int_t^{\cdot}Z'_rb^{\bar{q}(X)}_rdr$, and therefore
	 $$\tilde{Y}_s = \xi +\int_s^T\left(f^{\bar{q}(X)}_r+Z'_rb^{\bar{q}(X)}_r\right)dr -\int_s^TZ'_rdX_r + (M'_T-M'_s),\, s\in[t,T],$$ 
	 which implies that $\tilde{Y}=\bar{Y}^{t,x,G(\mathbbm{P})},\quad G(\mathbbm{P})$-a.s., by uniqueness of the solution of a BSDE. 
	 In particular, 
	 $$
	 \mathbbm{E}^{\mathbbm{P}}\!\!\left[\xi+\int_t^T\!\!\!\!f^Q_rdr\right]
	 =
	 \mathbbm{E}^{\mathbbm{P}}\!\!\left[\xi+\int_t^T\!\!\!\!f^{\bar{q}(X)}_rdr\right]
	 =
	 \mathbbm{E}^{\mathbbm{P}}[\tilde{Y}_t]  
	 = 
	 \mathbbm{E}^{G(\mathbbm{P})}[\tilde{Y}_t] 
	 = 
	 \mathbbm{E}^{G(\mathbbm{P})}[\bar{Y}^{t,x,G(\mathbbm{P})}_t].
	 $$
	By the comparison theorem for BSDEs (see Theorem 2.2 in \cite{el1997backward} for instance), and the definition of $F$ and $\hat{\sigma}^2$ we have that $\mathbbm{E}^{G(\mathbbm{P})}[\bar{Y}^{t,x,G(\mathbbm{P})}_t]\leq \mathbbm{E}^{G(\mathbbm{P})}[Y^{t,x,G(\mathbbm{P})}_t]$, and therefore the inequality in \eqref{EqGirs} holds.
\\
\\
{\bf 2.} We next prove the converse inequality $V_t(x)\geq \hat{\mathcal{Y}}_t(x)$. Recall the maximizer $\hat q$ introduced in Lemma \ref{qhat}, and denote $\hat{q}_r:=\hat{q}_r(X,Z^{t,x,\mathbbm{P}}_r,\hat{\sigma}^2_r)$. Then, we have for all $\mathbbm{P}\in\mathcal{M}_{t,x}$ that
\begin{equation}\label{EqGirs2}
	Y^{t,x,\mathbbm{P}}_s 
	= 
	\xi
	+ \int_s^T H_r\left(X,Z^{t,x,\mathbbm{P}}_r,\hat{q}_r\right)dr 
	                 - Z^{t,x,\mathbbm{P}}_r dX_r 
	                 + dM^{t,x,\mathbbm{P}}_r,~s\in[t,T],~ \mathbbm{P}-\text{a.s.}
\end{equation}
Proceeding as in the first part of this proof, we consider the change of measure $\frac{d\mathbbm{Q}}{d\mathbbm{P}}:=\mathcal{E}(\hat{L})$ where $ \hat{L}:=-\int_t^{\cdot} \int_{A\times B}\lambda_r(X,a)\hat{q}^A_r(da)dN^B(db,dr).$ As $\langle X, \hat{L}\rangle = -\int_t^{\cdot} b^{\hat{q}}_rdr$ $\mathbbm{P}$-a.s., it follows from \eqref{EqGirs2} that 
$\langle Y^{t,x,\mathbbm{P}}, \hat{L}\rangle = -\int_t^{\cdot}Z_r^{t,x,\mathbbm{P}}b^{\hat{q}}_rdr,$
and we conclude from the Girsanov Theorem that $ Y^{t,x,\mathbbm{P}}$ is a $\mathbbm{Q}$-martingale.
Finally, let $\hat{G}(\mathbbm{P}):=\mathbbm{Q}\circ(X,\hat{q})^{-1}$. By construction, $\hat{G}(\mathbbm{P})$ belongs to $\overline{\mathcal{P}}_{t,x}$, and we have
 $$
	\mathbbm{E}^{\hat{G}(\mathbbm{P})}\left[\xi+\int_t^Tf^Q_rdr\right] = \mathbbm{E}^{\mathbbm{Q}}\left[\xi+\int_t^Tf^{\hat{q}}_rdr\right] =\mathbbm{E}^{\mathbbm{Q}}[Y^{t,x,\mathbbm{P}}_t] = \mathbbm{E}^{\mathbbm{P}}[Y^{t,x,\mathbbm{P}}_t].
 $$
By the arbitrariness of $\mathbbm{P}\in\mathcal{M}_{t,x}$, and the fact that $\hat{G}(\mathbbm{P})$ belongs to $\overline{\mathcal{P}}_{t,x}$, this implies that $V_t(x)\geq\hat{\mathcal{Y}}_t(x)$.
\end{proof}
\begin{prooff}{} {\bf \hspace{-5mm} of Theorem \ref{Th2BSDE}} By the previous proposition, we have that  $V=\hat{\mathcal{Y}}$. Moreover, $(t,x)\mapsto V_t(x)$ is continuous by Proposition \ref{Vcontinuous}, so $t\mapsto \hat{\mathcal{Y}}_t(X_{\wedge t})$ is a continuous process.
	The present theorem therefore follows from Theorem 4.6 in \cite{soner_touzi_zhang2}  or Section 4.4 of \cite{ptz}, where we do not have to consider the path regularization of $t\mapsto \hat{\mathcal{Y}}_t(X_{\wedge t})$ as we have shown that it is continuous in the present setup.
\end{prooff}

\section{Proof of Theorem \ref{ThMkV}}\label{S6}
We will make use of Theorem \ref{T_existence} in a setup with no common noise. In particular, with the notations of Section \ref{SMFG}, we have $p_0=0$, $W=W^1$ and $M$ is deterministic. 

By Theorem \ref{T_existence}, there exist $m\in \mathfrak{M}_+^1(\mathcal{X})$ and $\widehat{\mathbbm{P}}^*\in \mathfrak{M}_+^1(\mathcal{X}\times\mathcal{Q}\times\mathcal{W})$ which maximizes $\mathbbm{E}^{\mathbbm{P}}[\xi+\int_0^Tf^Q_rdr]$ within all elements $\mathbbm{P}\in\mathfrak{M}_+^1(\mathcal{X}\times\mathcal{Q}\times\mathcal{W})$ satisfying Definition \ref{def_admissible} Item 1, with $m$ replacing $M$, and such that $\widehat{\mathbbm{P}}^*\circ X^{-1}=m$.

Let $\mathbbm{P}^*:=\widehat{\mathbbm{P}}^*\circ(X,Q)^{-1}$. We have  $m=\mathbbm{P}^*\circ X^{-1}$ and $\mathbbm{P}^*\in\overline{\mathcal{P}}^{m}$. In particular $m\in\mathcal{P}^m$, as required in Definition \ref{MkV2BSDE}.

We remark that $\xi,f$ do not depend in $W$. 
For any $\mathbbm{Q}\in \overline{\mathcal{P}}^{m}$, there exists $\widehat{\mathbbm{Q}}\in\mathfrak{M}_+^1(\mathcal{X}\times\mathcal{Q}\times\mathcal{W})$ satisfying Definition \ref{def_admissible} Item 1. and such that  $\mathbbm{Q}:=\widehat{\mathbbm{Q}}\circ(X,Q)^{-1}$, hence such that 
 \begin{eqnarray*}
 \mathbbm{E}^{\mathbbm{Q}}\left[\xi+\int_0^Tf^Q_rdr\right] 
 &=& 
 \mathbbm{E}^{\widehat{\mathbbm{Q}}}\left[\xi+\int_0^Tf^Q_rdr\right] 
 \\
 &\le& 
 \mathbbm{E}^{\widehat{\mathbbm{P}}^*}\left[\xi+\int_0^Tf^Q_rdr\right] 
 \;=\; \mathbbm{E}^{\mathbbm{P}^*}\left[\xi+\int_0^Tf^Q_rdr\right] .
 \end{eqnarray*}
This shows that $m=\mathbbm{P}^*\circ X^{-1}$ and
	\begin{equation}\label{E1S5}
		 V^{m}_0(0)
		\;=\;
		\mathbbm{E}^{\mathbbm{P}^*}\left[\xi+\int_0^Tf^Q_rdr\right]
		\;=\;
		\sup_{\mathbbm{P}\in\overline{\mathcal{P}}^{m}}
		\mathbbm{E}^{\mathbbm{P}}\left[\xi+\int_0^Tf^Q_rdr\right],
	\end{equation}
meaning that $m$ is a solution of the Mean-Field game on the restricted canonical space $\Omega=\mathcal{X}\times\mathcal{Q}$.

We set $Y_t=V^{m}_t(X_{\wedge t}),$ $t\in[0,T]$. By Theorem \ref{Th2BSDE}, $Y\in \mathbbm{S}^2(\mathcal{P}^m)$ and  there exists a process $Z\in\mathbbm{H}^2(\mathcal{P}^{m})$, such that  the process $U$ defined by
	\begin{equation}\label{P2S5}
	U := Y_{\cdot} - Y_0 +\int_0^{\cdot}F_r(Z_r,\hat{\sigma}_r^2,m)dr - \int_0^{\cdot}Z_rdX_r
	\end{equation}
	is a c\`adl\`ag $\mathbbm{P}$-supermartingale orthogonal to $X$ for all $\mathbbm{P}\in\mathcal{P}^{m}$. Consider the Doob-Meyer decomposition of the $m-$supermartingale $U=M-K$ into an $m$-martingale $M$ orthogonal to $X$, and an $m$-a.s. nondecreasing process $K$.
We define $\bar{q},N^B,L$ and $G(\mathbbm{P}^*)$ as in the proof of Proposition \ref{PGirs2}.
	Since $M$ is orthogonal to $X$ then $N^B$ can be taken orthogonal to $M$ (see Proposition III-9 in \cite{EK_Mele}) so $L$ is orthogonal to $M$. By the Girsanov Theorem, $M$ is also a $G(\mathbbm{P}^*)$-martingale. Then, it follows from \eqref{P2S5} that $(Y,Z)$ solves the BSDE
	$$
	Y_t = \xi +\int_t^T F_r(Z_r,\hat{\sigma}_r^2,m)dr+dK_r-Z_rdX_r - dM_r,\, t\in[0,T],\, G(\mathbbm{P}^*)-\text{a.s}.
	$$
with orthogonal martingale $M$. As $K$ is  $G(\mathbbm{P}^*)$-a.s. non-decreasing and positive, we have by the standard comparison result of BSDEs that $Y_0 \ge \mathbbm{E}^{G(\mathbbm{P}^*)}[Y^{G(\mathbbm{P}^*)}_0]$, where $(Y^{\mathbbm{P}^*},Z^{\mathbbm{P}^*})$ is defined as in \eqref{BSDE-P} by
	$$
	Y^{\mathbbm{P}^*}_t 
	= 
	\xi + \int_t^T F_r(Z^{\mathbbm{P}^*}_r,\hat{\sigma}_r^2,m)dr 
	                     - Z^{\mathbbm{P}^*}_rdX_r 
	                     - dM^{\mathbbm{P}^*}_r,\quad t\in[0,T],\quad \mathbbm{P^*}-\text{a.s.}
	$$
Moreover, the requirement that $U$ is an $m-$martingale is equivalent to $K\equiv 0$, $G(\mathbbm{P}^*)-$a.s. which is in turn equivalent to $Y_0 = \mathbbm{E}^{G(\mathbbm{P}^*)}[Y^{G(\mathbbm{P}^*)}_0]$, which we prove. As $m$ satisfies \eqref{E1S5}, it follows from \eqref{EqGirs} and Proposition \ref{PGirs2} that 
	\begin{equation}\label{E2S5}
	Y_0 
	\;=\; 
	V^{m}_0(0)
	\;=\;
	\mathbbm{E}^{\mathbbm{P}^*}\left[\xi+\int_0^Tf^Q_rdr\right]
	\;\le\; 
	\mathbbm{E}^{G(\mathbbm{P}^*)}[Y^{G(\mathbbm{P}^*)}_0],
	\end{equation} 
	and the required result follows from the fact that
	$$
	Y_0 = V^{m}_0(0)=\underset{\mathbbm{P}\in\overline{\mathcal{P}}^{m}}{\text{max} }\mathbbm{E}^{\mathbbm{P}}\left[\xi+\int_0^Tf^Q_rdr\right]	= \underset{\mathbbm{P}\in \mathcal{M}^{m}}{\text{sup}}\mathbbm{E}^{\mathbbm{P}}[Y^{\mathbbm{P}}_0]\geq \mathbbm{E}^{G(\mathbbm{P}^*)}[Y^{G(\mathbbm{P}^*)}_0].
	$$ 
\ep

\begin{appendix}
	
\section{Basic results concerning correspondences}

\begin{definition}
	Let $E,F$ be two Hausdorff topological spaces. A mapping $T$ from $E$ into the subsets of $F$ is called a correspondence from $E$ into $F$, which we summarize with the notation $T:E\xtwoheadrightarrow{  }F$.
	
	$T$ is called upper hemicontinuous (in short uhc) if for every $x\in E$ and  any neighborhood $U$ of $T(x)$,
	there is a neighborhood $V$ of $x$ such that $z \in V$ implies $T(z) \subset U$.
	
	$T$ is lower hemicontinuous (in short lhc) if for every $x\in E$ and any open set $U$ that meets $T(x)$  there is a neighborhood $V$ of $x$ such that $z \in V$ implies $T(z) \cap U \neq \emptyset$.
	
	We say that $T$ is continuous if it is both uhc and lhc. Finally, $T$ is said to have closed graph, if its graph $Gr(T):=\{(x,y):x\in E, y\in T(x)\}$ is a closed subset of $E\times F$.
\end{definition}

We collect in the following Proposition some classical results which can be found in \cite{aliprantis} see Theorems 17.10, 17.11, 17.15, 17.23 and Lemma 17.8.

\begin{proposition}\label{alipran}
	\begin{enumerate}\
		\item If $T$ is an uhc correspondence with compact values, then it has closed graph;
		\item conversely, if $T$ has closed graph and $F$ is compact, then $T$ is uhc;
		\item if $F$ is a metric space and $T$ is compact valued, then $T$ may be seen as a function from $E$ to $\rm{Comp}(F)$ the set of non-empty compact sets of $F$, which may be equipped with a metric called the Hausdorff metric such that $T$ is continuous as a correspondence iff it is continuous as a function for that metric;
		\item the composition of uhc (resp. lhc, continuous) correspondences is uhc (resp. lhc, continuous);
		\item the image of a compact set under
		a compact-valued uhc correspondence is compact.
	\end{enumerate}
\end{proposition}

We now recall the Berge maximum theorem (see Theorem 17.31 in \cite{aliprantis}).
\begin{theorem}\label{Berge}
	Let $T:E\xtwoheadrightarrow{  } F$ be a continuous nonempty compact valued correspondence between topological spaces. Let $J:F\longrightarrow \mathbbm{R}$ be a continuous function, then the correspondence $T^*:E\xtwoheadrightarrow{  } F$ defined for all $x\in E$ by $$T^*(x):=\underset{y\in T(x)}{\text{\rm{Argmax} }} J(y),$$ is uhc and nonempty compact valued.
	
	Moreover, the mapping $m:E\rightarrow \mathbbm{R}$ given for all $x\in E$ by 
	$$m(x):=\underset{y\in T(x)}{\text{\rm{max} }} J(y),$$ is continuous.
\end{theorem}

In \cite{horvath}, Horvath extended the $\epsilon$-approximate selection Theorem obtained by Cellina in \cite{cellina}. Although it was stated in a framework of generalized convex structures, the Theorem 6 of \cite{horvath} and the lines after its proof imply the following.
\begin{assumption}\label{HypE}
	$E$ is a subset of a locally convex topological vector space, such that there exists a distance $d_E$ metrizing the induced topology of $E$ and such that all open balls are convex, and that any neighborhood $\{y\in E:d_E(y,C)<r\}$ of a convex set $C$ is convex.
\end{assumption}

\begin{theorem}\label{Cellina}
	Let $(K,d_K)$ be a compact metric space and $(E,d_E)$ verifying Assumption \ref{HypE}.  We denote by $d$ the distance $d_K+d_E$ on $K\times E$.

	Let $T$ be an uhc correspondence taking nonempty compact convex values from $K$ to $E$, then for any $\epsilon>0$, there exists a continuous function $f_{\epsilon}:K\longrightarrow E$ such that for all $x\in K$,
	$$d((x,f_{\epsilon}(x)),Gr(T)):=\inf\{d((x,f_{\epsilon}(x)),(y,z):y\in E,z\in T(y)\}<\epsilon.$$
\end{theorem}

The following theorem is a generalization of Kakutani's Theorem adapted from Proposition 7.4 in \cite{LackerPathDep} which itself adapts a result of Cellina, see Theorem 1 in \cite{cellina}.

	\begin{theorem}\label{Kakutani}
		Let $(K,d)$ be a compact convex subset of a locally convex topological vector space, $(E,d_E)$ verifying Assumption \ref{HypE}, $T$ be an uhc correspondence taking nonempty compact convex values from $K$ to $E$ and $\phi$ be a continuous function from $E$ to $K$.
		
		Then there exists some $x\in K$ such that $x\in\phi\circ T(x)$.
	\end{theorem}
	
	\begin{proof}
		Let $Gr(T) := \{(x, y) \in K \times E : y \in T(x)\}$. By previous Theorem \ref{Cellina}, for every $n\in\mathbbm{N}$, there exists a continuous $f_n:K\longrightarrow E$ such that for all $x\in K$,
		$$\inf\{d((x,f_{n}(x)),Gr(T)\}<\frac{1}{n}.$$
		Since $\phi\circ f_n:K\longrightarrow K$ is continuous, there exists by Schauder’s fixed point theorem some $x_n \in K$ such
		that $x_n = \phi(f_n(x_n))$. By Proposition \ref{alipran} Items 1 and 5, since $T$ is uhc and compact valued then  $T(K) :=\underset{x\in K}{\bigcup}T(x)$ is compact and  $Gr(T)$ is closed. Thus $Gr(T) \subset K \times T(K)$ is
		compact. Since $d((x_n , f_n(x_n)), Gr(T)) \longrightarrow 0$ and $Gr(T)$ is compact, there exists a
		subsequence $x_{n_k}$ and a point $(x, y) \in Gr(T)$ such that  $(x_{n_k},f_{n_k}(x_{n_k}))\longrightarrow (x,y)$.
		Now by continuity of $\phi$ we have
		$$x = \text{lim }x_{n_k} = \text{lim } \phi(f_{n_k}(x_{n_k})) = \phi(y),$$ with $y\in T(x)$ so the proof is complete.
	\end{proof}

\begin{lemma}\label{Lkakutani}
	Let $S$ be a polish space and $E$ be a  convex subset of $\mathfrak{M}_+^1(S)$, equipped with the topology of weak convergence, then there exists on $E$ a distance $d_E$ such that $(E,d_E)$ verifies Assumption \ref{HypE}.
	
	In particular, Theorem \ref{Kakutani} applies for such a choice of space $E$.
\end{lemma}

\begin{proof}
	It is immediate that in a normed space, the distance induced by the norm satisfies  Assumption \ref{HypE}. This implies that if we consider a convex subset $E$ of a normed space $(F,\|\cdot\|)$ and equip $E$  with the distance $d_E$ defined by $d_E(x,y):=\|x-y\|$ then $(E,d_E)$ verifies Assumption \ref{HypE}.
	
	We now recall that $\mathfrak{M}_+^1(S)$ is a convex subset of the vector space $\mathfrak{M}(S)$ which can be equipped with the Kantorovic-Rubinshtein norm (see Section 8.3 in \cite{bogachev2} for an introduction) and that on $\mathfrak{M}_+^1(S)$, that norm induces the topology of weak convergence, see Theorem 8.3.2 in \cite{bogachev2}. This concludes the proof. 
\end{proof}

\end{appendix}

\bibliographystyle{plain}
\bibliography{../../../biblioPhDBarrasso}

\end{document}